\documentclass[12pt,reqno]{amsart}%
\usepackage{amsfonts}
\usepackage{amsmath}
\usepackage{amssymb}
\usepackage{graphicx}%
\providecommand{\U}[1]{\protect\rule{.1in}{.1in}}
\newtheorem{theorem}{Theorem}[section]
\theoremstyle{plain}

\newtheorem{lemma}{Lemma}[section]

\numberwithin{equation}{section}
\begin{document}
\title[Critical and subcritical Moser-Trudinger-Adams inequalities ]{Equivalence of critical and subcritical sharp Trudinger-Moser-Adams
inequalities }
\author{Nguyen Lam}
\address{Department of Mathematics\\
University of Pittsburgh\\
Pittsburgh, PA 15260, USA}
\email{nhlam@pitt.edu}
\author{Guozhen Lu}
\address{Department of Mathematics\\
Wayne State University\\
Detroit, MI 48202, USA}
\email{gzlu@wayne.edu}
\author{Lu Zhang}
\address{Department of Mathematics\\
Wayne State University\\
Detroit, MI 48202, USA}
\email{eu4347@wayne.edu}
\thanks{Research of this work was partly supported by a US NSF grant DMS\#1301595.}

\begin{abstract}
Sharp Trudinger-Moser inequalities on the first order Sobolev spaces and their
analogous Adams inequalities on high order Sobolev spaces play an important
role in geometric analysis, partial differential equations and other branches
of modern mathematics. Such geometric inequalities have been studied
extensively by many authors in recent years and there is a vast literature.
There are two types of such optimal inequalities: critical and subcritical
sharp inequalities, both are with best constants. Critical sharp inequalities
are under the restriction of the full Sobolev norms for the functions under
consideration, while the subcritical inequalities are under the restriction of
the partial Sobolev norms for the functions under consideration. There are
subtle differences between these two type of inequalities. Surprisingly, we
prove in this paper that these critical and subcritical Trudinger-Moser and
Adams inequalities are actually equivalent. Moreover, we also establish the
asymptotic behavior of the supremum for the subcritical Trudinger-Moser and
Adams inequalities on the entire Euclidean spaces (Theorem 1.1 and Theorem
1.3) and provide a precise relationship between the supremums for the critical
and subcritical Trudinger-Moser and Adams inequalities (Theorem 1.2 and
Theorem 1.4). Since the critical Trudinger-Moser and Adams inequalities can be
easier to prove than subcritical ones in some occasions, and more difficult to
establish in other occasions, our results and the method  suggest a new
approach to both the critical and subcritical Trudinger-Moser and Adams type inequalities.

\end{abstract}
\maketitle


\section{Introduction}

In this section, we will begin with giving an overview of the state of affairs
of the best constants for sharp Trudinger and Adams inequalities. Subsection
1.1 concerns the sharp Trudinger-Moser inequalities and Subsection 1.2
discusses the sharp Adams inequalities involving high order derivatives. In
Subection 1.3, we will state our main results on the equivalence between
critical and subcritical Trudinger-Moser and Adams inequalities.

\subsection{Trudinger-Moser inequality}

Motivated by the applications to the prescribed Gauss curvature problem on two
dimensional sphere $\mathbb{S}^{2}$, J. Moser proved in \cite{Mo} an
exponential type inequality on $\mathbb{S}^{2}$ with an optimal constant. In
the same paper, he sharpened an inequality on any bounded domain $\Omega$ in
the Euclidean space $\mathbb{R}^{N}$ studied independently by Pohozaev
\cite{Po}, Trudinger \cite{Tru} and Yudovich \cite{Yu}, namely the embedding
$W_{0}^{1,N}\left(  \Omega\right)  \subset L_{\varphi_{N}}\left(
\Omega\right)  $, where $L_{\varphi_{N}}\left(  \Omega\right)  $ is the Orlicz
space associated with the Young function $\varphi_{N}(t)=\exp\left(
\alpha\left\vert t\right\vert ^{N/(N-1)}\right)  -1$ for some $\alpha>0$. More
precisely, using the Schwarz rearrangement, Moser proved the following
inequality in \cite{Mo}:

\vskip 0.5cm

\textbf{Theorem A. }\textit{Let }$\Omega$\textit{ be a domain with finite
measure in Euclidean }$N-$\textit{space }$%
\mathbb{R}
^{N},~n\geq2$\textit{. Then there exists a constant }$\alpha_{N}>0$\textit{,
such that }%
\begin{equation}
\frac{1}{\left\vert \Omega\right\vert }\int_{\Omega}\exp\left(  \alpha
_{N}\left\vert u\right\vert ^{\frac{N}{N-1}}\right)  dx\leq c_{0} \label{1.0}%
\end{equation}
\textit{for any }$u\in W_{0}^{1,N}\left(  \Omega\right)  $\textit{ with }%
$\int_{\Omega}\left\vert \nabla u\right\vert ^{N}dx\leq1$\textit{.}
\textit{The constant } $\alpha_{N}=\omega_{N-1}^{\frac{1}{N-1}}$,
\textit{where} $\omega_{N-1}$ \textit{is the area of the surface of the unit
}$N-$ \textit{ball, is optimal in the sense that if we replace } $\alpha_{N}$
\textit{by any number }$\alpha>\alpha_{N}$, \textit{then the above inequality
can no longer hold with some} $c_{0}$ \textit{independent of} $u$.

\vskip0.5cm

Moser used the following symmetrization argument: every function $u$ is
associated to a radially symmetric function $u^{\ast}$ such that the
sublevel-sets of $u^{\ast}$ are balls with the same area as the corresponding
sublevel-sets of $u$. Moreover, $u$ is a positive and non-increasing function
defined on $B_{R}\left(  0\right)  $ where $\left\vert B_{R}\left(  0\right)
\right\vert =\left\vert \Omega\right\vert $. Hence, by the layer cake
representation, we can have that%
\[
\int_{\Omega}f\left(  u\right)  dx=\int_{B_{R}\left(  0\right)  }f\left(
u^{\ast}\right)  dx
\]
for any function $f$ that is the difference of two monotone functions. In
particular, we obtain%
\begin{align*}
\left\Vert u\right\Vert _{p}  &  =\left\Vert u^{\ast}\right\Vert _{p};\\
\int_{\Omega}\exp\left(  \alpha\left\vert u\right\vert ^{\frac{n}{n-1}%
}\right)  dx  &  =\int_{B_{R}\left(  0\right)  }\exp\left(  \alpha\left\vert
u^{\ast}\right\vert ^{\frac{n}{n-1}}\right)  dx.
\end{align*}
Moreover, the well-known P\'{o}lya-Szeg\"{o} inequality
\begin{equation}
\int_{B_{R}\left(  0\right)  }\left\vert \nabla u^{\ast}\right\vert ^{p}%
dx\leq\int_{\Omega}\left\vert \nabla u\right\vert ^{p}dx \label{1.2}%
\end{equation}
plays a crucial role in the approach of J. Moser.

\vskip0.5cm

As far as the existence of extremal functions of Moser's inequality, the first
breakthrough was due to the celebrated work of Carleson and Chang \cite{CC} in
which they proved that the supremum
\[
\sup_{u\in W_{0}^{1,N}\left(  \Omega\right)  , \int_{\Omega}\left\vert \nabla
u\right\vert ^{N}dx\leq1}\frac{1}{\left\vert \Omega\right\vert }\int_{\Omega
}\exp\left(  \alpha_{N} \left\vert u\right\vert ^{\frac{N}{N-1}}\right)  dx
\]
can be achieved when $\Omega$ is an Euclidean ball. This result came as a
surprise because it has been known that the Sobolev inequality does not have
extremal functions supported on any finite ball. Subsequently, existence of
extremal functions has been established on arbitrary domains in \cite{Flu},
\cite{Lin},  and on Riemannian manifolds in \cite{L1}, etc.

\vskip0.5cm

We note when the volume of $\Omega$ is infinite, the Trudinger-Moser
inequality (\ref{1.0}) becomes meaningless. Thus, it becomes interesting and
nontrivial to extend such inequalities to unbounded domains. Here we state the
following two such results in the Euclidean spaces.

\vskip0.5cm

We first recall the subcritical Moser-Trudinger inequality in the Euclidean
spaces established by Adachi and Tanaka \cite{AT}.

\vskip0.5cm

\textbf{Theorem B. }\textit{For any }$\alpha<\alpha_{N},~$\textit{there exists
a positive constant }$C_{N,\alpha}$\textit{ such that }$\forall u\in
W^{1,N}\left(
\mathbb{R}
^{N}\right)  ,~\left\Vert \nabla u\right\Vert _{N}\leq1:$\textit{ }%
\begin{equation}
\int_{%
\mathbb{R}
^{N}}\phi_{N}\left(  \alpha\left\vert u\right\vert ^{\frac{N}{N-1}}\right)
dx\leq C_{N,\alpha}\left\Vert u\right\Vert _{N}^{N}, \label{1.3}%
\end{equation}
\textit{where}%
\[
\phi_{N}(t)=e^{t}-%
{\displaystyle\sum\limits_{j=0}^{N-2}}
\frac{t^{j}}{j!}.
\]
\textit{The constant }$\alpha_{N}$\textit{ is sharp in the sense that the
supremum is infinity when }$\alpha\geq\alpha_{N}.$

\vskip0.5cm

We note in the above theorem, we only impose the restriction on the norm
$\int_{\mathbb{R}^{N}}\left\vert \nabla u\right\vert ^{N}$ without restricting
the full norm
\[
\left[  \int_{\mathbb{R}^{N}}\left\vert \nabla u\right\vert ^{N}+\tau
\int_{\mathbb{R}^{N}}\left\vert u\right\vert ^{N}\right]  ^{1/N}\leq1.
\]
The method in \cite{AT} requires a symmetrization argument which is not
available in many other non-Euclidean settings. The above inequality fails at
the critical case $\alpha=\alpha_{N}$. So it is natural to ask when the above
can be true when $\alpha=\alpha_{N}$. This is done by Ruf \cite{R} and Li and
Ruf \cite{LR} by using the restriction of the full norm of the Sobolev space
$W^{1,N}\left(  \mathbb{R}^{N}\right)  :$ $\left[  \int_{\mathbb{R}^{N}%
}\left\vert \nabla u\right\vert ^{N}+\tau\int_{\mathbb{R}^{N}}\left\vert
u\right\vert ^{N}\right]  ^{1/N}.$

\vskip0.5cm

\textbf{Theorem C. }\textit{For all }$0\leq\alpha\leq\alpha_{N}:$%
\begin{equation}
\underset{\left\Vert u\right\Vert \leq1}{\sup}\int_{%
\mathbb{R}
^{N}}\phi_{N}\left(  \alpha\left\vert u\right\vert ^{\frac{N}{N-1}}\right)
dx<\infty\label{1.4}%
\end{equation}
\textit{where }%
\[
\left\Vert u\right\Vert =\left(  \int_{%
\mathbb{R}
^{N}}\left(  \left\vert \nabla u\right\vert ^{N}+\left\vert u\right\vert
^{N}\right)  dx\right)  ^{1/N}.
\]
\textit{Moreover, this constant }$\alpha_{N}$\textit{ is sharp in the sense
that if }$\alpha>\alpha_{N}$\textit{, then the supremum is infinity.}\textbf{
}

\vskip0.5cm

\vskip0.5cm

Sharp critical and subcritical Trudinger-Moser inequalities on infinite volume
domains of the Heisenberg groups were also established in \cite{LaLu6,
LaLuTa2} by using a symmetrization-free method.

\medskip

The inequality (\ref{1.3}) uses the seminorm $\left\Vert \nabla u\right\Vert
_{N}$ and hence fails at the critical case $\alpha=\alpha_{N}$, the best
constant. Thus, it can be considered as a sharp subcritical Trudinger-Moser
inequality. In (\ref{1.4}), when using the full norm of $W^{1,N}\left(
\mathbb{R}
^{N}\right)  $, the best constant could be attained. Namely, the inequality
holds at the critical case $\alpha=\alpha_{N}$. Hence, (\ref{1.4}) is the
sharp critical Trudinger-Moser inequality.

\vskip0.5cm

Nevertheless, the main purpose of this paper is to show that in fact, these
two versions of critical and subcritical Trudinger-Moser type inequalities are
indeed equivalent. Since the critical Trudinger-Moser type inequality is
easier to study than the subcritical one in some occasions, and it is easier
to investigate subcritical Trudinger-Moser type inequality than the critical
one in other occasions, our paper suggests a new approach to both the critical
and subcritical Trudinger-Moser type inequalities.

\subsection{Adams inequalities}

It is worthy noting that symmetrization has been a very useful and efficient
(and almost inevitable) method when dealing with the sharp geometric
inequalities. Thus, it is very fascinating to investigate such sharp geometric
inequalities, in particular, the Trudinger-Moser type inequalities, in the
settings where the symmetrization is not available such as on the higher order
Sobolev spaces, the Heisenberg groups, Riemannian manifolds, sub-Riemannian
manifolds, etc. Indeed, in these settings, an inequality like (\ref{1.2}) is
not available. In these situations, the first break-through came from the work
of D. Adams \cite{A} when he attempted to set up the Trudinger-Moser
inequality in the higher order setting in Euclidean spaces. In fact, using a
new idea that one can write a smooth function as a convolution of a (Riesz)
potential with its derivatives, and then one can use the symmetrization for
this convolution, instead of the symmetrization of the higher order
derivatives, Adams proved the following inequality with boundary Dirichlet
condition \cite{A} which was extended to the Navier boundary condition in
\cite{Ta} when $\beta=0$, and then the first two authors extended it to the
case $0\le\beta<N$ \cite{LaLu5}. The following is taken from \cite{LaLu5}.

\vskip0.5cm

\textbf{Theorem D.} \textit{Let }$\Omega$\textit{ be an open and bounded set
in }$\mathbb{R}^{N}$\textit{. If }$m$ \textit{is a positive integer less than
}$N$, $0\leq\beta<N$, \textit{then there exists a constant} $C_{0}%
=C(N,m,\beta)>0$ \textit{such that for any} $u\in W_{N}^{m,\frac{N}{m}}%
(\Omega)$ \textit{and} $||\nabla^{m}u||_{L^{\frac{N}{m}}(\Omega)}\leq1$,
\textit{then}
\[
\frac{1}{|\Omega|^{1-\frac{\beta}{N}}}\int_{\Omega}\exp(\alpha\left(
1-\frac{\beta}{N}\right)  |u(x)|^{\frac{N}{N-m}})\frac{dx}{\left\vert
x\right\vert ^{\beta}}\leq C_{0}%
\]
\textit{for all} $\beta\leq\beta(N,m)$ \textit{where}
\[
\beta(N,\ m)\ =\left\{
\begin{array}
[c]{c}%
\frac{N}{w_{N-1}}\left[  \frac{\pi^{N/2}2^{m}\Gamma(\frac{m+1}{2})}%
{\Gamma(\frac{N-m+1}{2})}\right]  ^{\frac{N}{N-m}}\text{ }\mathrm{when}%
\,\,\,m\,\,\mathrm{is\,\,odd}\\
\frac{N}{w_{N-1}}\left[  \frac{\pi^{N/2}2^{m}\Gamma(\frac{m}{2})}{\Gamma
(\frac{N-m}{2})}\right]  ^{\frac{N}{N-m}}\text{ \ }\mathrm{when}%
\,\,\,m\,\,\mathrm{is\,\,even}%
\end{array}
\right.  .
\]
\textit{ Furthermore, the constant } $\beta(N,m)$ \textit{is optimal in the
sense that for any }$\mathit{\alpha>\beta(N,m)}$\textit{, the integral can be
made as large as possible.}\newline

\vskip0.5cm

Adams inequalities have been extended to compact Riemannian manifolds in
\cite{Fontana}. The Adams inequalities with optimal constants for high order
derivatives on domains of infinite volume were recently established by Ruf and
Sani in \cite{RS} in the case of even order derivatives and by Lam and Lu for
all order of derivatives including fractional orders \cite{LaLu4, LaLu7}. The
idea of \cite{RS} is to use the comparison principle for polyharmonic
equations (thus could deal with the case of even order of derivatives) and
thus involves some difficult construction of auxiliary functions. The argument
in \cite{LaLu4, LaLu7} uses the representation of the Bessel potentials and
thus avoids dealing with such a comparison principle. In particular, the
method developed in \cite{LaLu7} adapts the idea of deriving the sharp
Moser-Trudinger-Adams inequalities on domains of finite measure to the entire
spaces using the level sets of the functions under consideration. Thus, the
argument in \cite{LaLu7} does not use the symmetrization method and thus also
works for the sub-Riemannian setting such as the Heisenberg groups
\cite{LaLu6, LaLuTa2}. The following general version is taken from
\cite{LaLu7}.

\medskip

\textbf{Theorem (Lam-Lu, 2013)} \label{frac2}Let $0<\gamma<n$ be an arbitrary
real positive number, $p=\frac{n}{\gamma}$ and $\tau>0$. There holds%
\[
\underset{u\in W^{\gamma,p}\left(
\mathbb{R}
^{n}\right)  ,\left\Vert \left(  \tau I-\Delta\right)  ^{\frac{\gamma}{2}%
}u\right\Vert _{p}\leq1}{\sup}\int_{%
\mathbb{R}
^{n}}\phi\left(  \beta_{0}\left(  n,\gamma\right)  \left\vert u\right\vert
^{p^{\prime}}\right)  dx<\infty
\]
where
\begin{align*}
\phi(t)  &  =e^{t}-%
{\displaystyle\sum\limits_{j=0}^{j_{p}-2}}
\frac{t^{j}}{j!},\\
j_{p}  &  =\min\left\{  j\in%
\mathbb{N}
:j\geq p\right\}  \geq p.
\end{align*}
Furthermore this inequality is sharp in the sense that if $\beta_{0}\left(
n,\gamma\right)  $ is replaced by any $\beta>\beta_{0}\left(  n,\gamma\right)
$, then the supremum is infinite.

\medskip

Very little is known for existence of extremals for Adams inequalities.
Existence of extremal functions for the Adams inequality on bounded domains in
Euclidean spaces has been established in \cite{LY} and compact Riemannian
manifolds by \cite{LN} only when $N=4$ and $m=2$ and is still widely open in
other cases.

\subsection{ Our Main Results}

Though the Adachi-Tanaka type inequality in unbounded domains has been known
for quite some time, it is still not known what the following supremum is:
\[
\sup_{\left\Vert \nabla u\right\Vert _{N}\leq1}\frac{1}{\left\Vert
u\right\Vert _{N}^{N-\beta}}\int_{%
\mathbb{R}
^{N}}\phi_{N}\left(  \alpha\left(  1-\frac{\beta}{N}\right)  \left\vert
u\right\vert ^{\frac{N}{N-1}}\right)  \frac{dx}{\left\vert x\right\vert
^{\beta}}.
\]
In particular, we do not even know how the supremum behaves asymptotically
when $\alpha$ goes to $\alpha_{N}$.

\vskip0.5cm

The following theorem answers this question and provides the lower and upper
bounds asymptotically for the supremum.

\begin{theorem}
\label{improvedAT}\textit{Let }$N\geq2$\textit{, }$\alpha_{N}=N\left(
\frac{N\pi^{\frac{N}{2}}}{\Gamma(\frac{N}{2}+1)}\right)  ^{\frac{1}{N-1}}%
$,$~0\leq\beta<N$ \textit{and }$0\leq\alpha<\alpha_{N}.$ Denote%
\[
AT\left(  \alpha,\beta\right)  =\sup_{\left\Vert \nabla u\right\Vert _{N}%
\leq1}\frac{1}{\left\Vert u\right\Vert _{N}^{N-\beta}}\int_{%
\mathbb{R}
^{N}}\phi_{N}\left(  \alpha\left(  1-\frac{\beta}{N}\right)  \left\vert
u\right\vert ^{\frac{N}{N-1}}\right)  \frac{dx}{\left\vert x\right\vert
^{\beta}}.
\]
Then there exist positive constants $c=c\left(  N,\beta\right)  $ and
$C=C\left(  N,\beta\right)  $ such that when $\alpha$ is close enough to
$\alpha_{N}:$
\begin{equation}
\frac{c\left(  N,\beta\right)  }{\left(  1-\left(  \frac{\alpha}{\alpha_{N}%
}\right)  ^{N-1}\right)  ^{\left(  N-\beta\right)  /N}}\leq AT\left(
\alpha,\beta\right)  \leq\frac{C\left(  N,\beta\right)  }{\left(  1-\left(
\frac{\alpha}{\alpha_{N}}\right)  ^{N-1}\right)  ^{\left(  N-\beta\right)
/N}}. \label{1.3.1}%
\end{equation}
Moreover, the constant $\alpha_{N}$ is sharp in the sence that $AT\left(
\alpha_{N},\beta\right)  =\infty.$
\end{theorem}

We note that we do not assume a priori the validity of the critical
Trudinger-Moser inequality with the restriction on the full norm (i.e., the
inequality (1.4)) in order to derive the above asymptotic behavior of the
supremum $AT(\alpha, \beta)$. We also mention that the upper bound in
(\ref{1.3.1}) in dimension two in the nonsingular case $\beta=0$ has also been
given in \cite{CST} using the sharp critical Trudinger-Moser inequality in
$\mathbb{R}^{2}$.

\vskip0.5cm

Next, we like to know how the supremum $AT(\alpha, \beta)$ we established in
Theorem \ref{improvedAT}  will provide a proof to the sharp critical
Trudinger-Moser inequality. Thus, this gives a new proof of the sharp critical
Trudinger-Moser inequality in all dimension $N$. We also answer the question
under for which $a$ and $b$ the critical Trudinger-Moser inequality holds
under the restriction of the full norm $\left\Vert \nabla u\right\Vert
_{N}^{a}+\left\Vert u\right\Vert _{N}^{b}\leq1$. Moreover, we establish the
precise relationship between the supremums for the critical and subcritical
Trudinger-Moser inequalities.

\begin{theorem}
\label{MT}\textit{Let }$N\geq2$\textit{,}$~0\leq\beta<N,~0<a,~b.$ Denote
\begin{align*}
MT_{a,b}\left(  \beta\right)   &  =\sup_{\left\Vert \nabla u\right\Vert
_{N}^{a}+\left\Vert u\right\Vert _{N}^{b}\leq1}\int_{%
\mathbb{R}
^{N}}\phi_{N}\left(  \alpha_{N}\left(  1-\frac{\beta}{N}\right)  \left\vert
u\right\vert ^{\frac{N}{N-1}}\right)  \frac{dx}{\left\vert x\right\vert
^{\beta}};\\
MT\left(  \beta\right)   &  =MT_{N,N}\left(  \beta\right)  .
\end{align*}
Then $MT_{a,b}\left(  \beta\right)  <\infty$ if and only if $b\leq N$. The
constant $\alpha_{N}$ is sharp. Moreover, we have the following identity:%
\begin{equation}
MT_{a,b}\left(  \beta\right)  =\sup_{\alpha\in\left(  0,\alpha_{N}\right)
}\left(  \frac{1-\left(  \frac{\alpha}{\alpha_{N}}\right)  ^{\frac{N-1}{N}a}%
}{\left(  \frac{\alpha}{\alpha_{N}}\right)  ^{\frac{N-1}{N}b}}\right)
^{\frac{N-\beta}{b}}AT\left(  \alpha,\beta\right)  . \label{1.3.2}%
\end{equation}
In particular, $MT\left(  \beta\right)  <\infty$ and
\[
MT\left(  \beta\right)  =\sup_{\alpha\in\left(  0,\alpha_{N}\right)  }\left(
\frac{1-\left(  \frac{\alpha}{\alpha_{N}}\right)  ^{N-1}}{\left(  \frac
{\alpha}{\alpha_{N}}\right)  ^{N-1}}\right)  ^{\frac{N-\beta}{N}}AT\left(
\alpha,\beta\right)  .
\]

\end{theorem}

We now consider the sharp subcritical and critical Adams inequalities on
$W^{2,\frac{N}{2}}\left(
\mathbb{R}
^{N}\right)  ,~N\geq3$. Our first result is the following sharp subcritical
Adams inequality which provides the asymptotic behavior of the supremum (lower
and upper bounds) in this case.

\begin{theorem}
\label{ATA}\textit{Let }$N\geq3$\textit{, }$0\leq\beta<N$ \textit{and }%
$0\leq\alpha<\beta\left(  N,2\right)  .$ Denote%
\begin{align*}
ATA\left(  \alpha,\beta\right)   &  =\sup_{\left\Vert \Delta u\right\Vert
_{\frac{N}{2}}\leq1}\frac{1}{\left\Vert u\right\Vert _{\frac{N}{2}}^{\frac
{N}{2}\left(  1-\frac{\beta}{N}\right)  }}\int_{%
\mathbb{R}
^{N}}\frac{\phi_{N,2}\left(  \alpha\left(  1-\frac{\beta}{N}\right)
\left\vert u\right\vert ^{\frac{N}{N-2}}\right)  }{\left\vert x\right\vert
^{\beta}}dx;\\
\phi_{N,2}\left(  t\right)   &  =%
{\displaystyle\sum\limits_{j\in\mathbb{N}:j\geq\frac{N-2}{2}}}
\frac{t^{j}}{j!}.
\end{align*}
Then there exist positive constants $c=c\left(  N,\beta\right)  $ and
$C=C\left(  N,\beta\right)  $ such that when $\alpha$ is close enough to
$\beta\left(  N,2\right)  :$
\begin{equation}
\frac{c\left(  N,\beta\right)  }{\left[  1-\left(  \frac{\alpha}{\beta\left(
N,2\right)  }\right)  ^{\frac{N-2}{2}}\right]  ^{1-\frac{\beta}{N}}}\leq
ATA\left(  \alpha,\beta\right)  \leq\frac{C\left(  N,\beta\right)  }{\left[
1-\left(  \frac{\alpha}{\beta\left(  N,2\right)  }\right)  ^{\frac{N-2}{2}%
}\right]  ^{1-\frac{\beta}{N}}}. \label{1.3.3}%
\end{equation}
Moreover, the constant $\beta\left(  N,2\right)  $ is sharp in the sence that
$AT\left(  \alpha_{N},\beta\right)  =\infty.$
\end{theorem}

The next theorem offers a precise relationship between the supremums of
critical and subcritical Adams inequalities. Thus, it also provides a new
approach of proving one of the critical and subcritical Adams inequalities
from the other.

\begin{theorem}
\label{A}Let $N\geq3$, $0\leq\beta<N,~0<a,~b.$ We denote:%
\begin{align*}
A_{a,b}\left(  \beta\right)   &  =\sup_{\left\Vert \Delta u\right\Vert
_{\frac{N}{2}}^{a}+\left\Vert u\right\Vert _{\frac{N}{2}}^{b}\leq1}\int_{%
\mathbb{R}
^{N}}\frac{\phi_{N,2}\left(  \beta\left(  N,2\right)  \left(  1-\frac{\beta
}{N}\right)  \left\vert u\right\vert ^{\frac{N}{N-2}}\right)  }{\left\vert
x\right\vert ^{\beta}}dx;\\
A_{\frac{N}{2},\frac{N}{2}}\left(  \beta\right)   &  =A\left(  \beta\right)  ;
\end{align*}
Then $A_{a,b}\left(  \beta\right)  <\infty$ if and only if $b\leq\frac{N}{2}$.
The constant $\beta\left(  N,2\right)  $ is sharp. Moreover, we have the
following identity:%
\begin{equation}
A_{a,b}\left(  \beta\right)  =\sup_{\alpha\in\left(  0,\beta\left(
N,2\right)  \right)  }\left(  \frac{1-\left(  \frac{\alpha}{\beta\left(
N,2\right)  }\right)  ^{\frac{N-2}{N}a}}{\left(  \frac{\alpha}{\beta\left(
N,2\right)  }\right)  ^{\frac{N-2}{N}b}}\right)  ^{\frac{N-\beta}{2b}%
}ATA\left(  \alpha,\beta\right)  . \label{1.3.4}%
\end{equation}
In particular, $A\left(  \beta\right)  <\infty$ and
\[
A\left(  \beta\right)  =\sup_{\alpha\in\left(  0,\beta\left(  N,2\right)
\right)  }\left(  \frac{1-\left(  \frac{\alpha}{\beta\left(  N,2\right)
}\right)  ^{\frac{N-2}{2}}}{\left(  \frac{\alpha}{\beta\left(  N,2\right)
}\right)  ^{\frac{N-2}{2}}}\right)  ^{\frac{N-\beta}{N}}ATA\left(
\alpha,\beta\right)  .
\]

\end{theorem}

Finally, we will study the following improved sharp critical Adams inequality
under the assumption that a version of the sharp subcritical Adams inequality
holds for the factional order derivatives:

\begin{theorem}
\label{generalA}Let $0<\gamma<N$ be an arbitrary real positive number,
$p=\frac{N}{\gamma}$,~$0\leq\alpha<\beta_{0}\left(  N,\gamma\right)  =\frac
{N}{\omega_{N-1}}\left[  \frac{\pi^{\frac{N}{2}}2^{\gamma}\Gamma\left(
\frac{\gamma}{2}\right)  }{\Gamma\left(  \frac{N-\gamma}{2}\right)  }\right]
^{\frac{p}{p-1}},~0\leq\beta<N,~0<a,~b$. We note
\begin{align*}
GATA\left(  \alpha,\beta\right)   &  =\sup_{u\in W^{\gamma,p}\left(
\mathbb{R}
^{N}\right)  :\left\Vert \left(  -\Delta\right)  ^{\frac{\gamma}{2}%
}u\right\Vert _{p}\leq1}\frac{1}{\left\Vert u\right\Vert _{p}^{p\left(
1-\frac{\beta}{N}\right)  }}\int_{%
\mathbb{R}
^{N}}\frac{\phi_{N,\gamma}\left(  \alpha\left(  1-\frac{\beta}{N}\right)
\left\vert u\right\vert ^{\frac{p}{p-1}}\right)  }{\left\vert x\right\vert
^{\beta}}dx;\\
GA_{a,b}\left(  \beta\right)   &  =\sup_{u\in W^{\gamma,p}\left(
\mathbb{R}
^{N}\right)  :\left\Vert \left(  -\Delta\right)  ^{\frac{\gamma}{2}%
}u\right\Vert _{p}^{a}+\left\Vert u\right\Vert _{p}^{b}\leq1}\int_{%
\mathbb{R}
^{N}}\frac{\phi_{N,\gamma}\left(  \beta_{0}\left(  N,\gamma\right)  \left(
1-\frac{\beta}{N}\right)  \left\vert u\right\vert ^{\frac{p}{p-1}}\right)
}{\left\vert x\right\vert ^{\beta}}dx
\end{align*}
where
\[
\phi_{N,\gamma}\left(  t\right)  =%
{\displaystyle\sum\limits_{j\in\mathbb{N}:j\geq p-1}}
\frac{t^{j}}{j!}.
\]
Assume that $GATA\left(  \alpha,\beta\right)  <\infty$ and there exists a
constant $C\left(  N,\gamma,\beta\right)  >0$ such that%
\begin{equation}
GATA\left(  \alpha,\beta\right)  \leq\frac{C\left(  N,\gamma,\beta\right)
}{\left(  1-\left(  \frac{\alpha}{\beta_{0}\left(  N,\gamma\right)  }\right)
^{\frac{p-1}{p}}\right)  } \label{1.3.5}%
\end{equation}
Then when $b\leq p$, we have $GA_{a,b}\left(  \beta\right)  <\infty.$ In
particular $GA_{p,p}\left(  \beta\right)  <\infty.$
\end{theorem}

Though we have to assume a sharp subcritical Adams inequality (\ref{1.3.5}),
the main idea of Theorem \ref{generalA} is that since $GATA\left(
\alpha,\beta\right)  $ is actually subcritical, i.e. $\alpha$ is strictly less
than the critical level $\beta_{0}\left(  N,\gamma\right)  $, it is easier to
study than $GA_{a,b}\left(  \beta\right)  $.\ Hence, it suggests a new
approach in the study of $GA_{a,b}\left(  \beta\right)  $.

\section{Some lemmata}

\begin{lemma}
\label{normalized1}%
\[
AT\left(  \alpha,\beta\right)  =\sup_{\left\Vert \nabla u\right\Vert _{N}%
\leq1;\left\Vert u\right\Vert _{N}=1}\int_{%
\mathbb{R}
^{N}}\phi_{N}\left(  \alpha\left(  1-\frac{\beta}{N}\right)  \left\vert
u\right\vert ^{\frac{N}{N-1}}\right)  \frac{dx}{\left\vert x\right\vert
^{\beta}}.
\]

\begin{proof}
For any $u\in W^{1,N}\left(
\mathbb{R}
^{N}\right)  :\left\Vert \nabla u\right\Vert _{N}\leq1$, we define%
\begin{align*}
v\left(  x\right)   &  =u\left(  \lambda x\right) \\
\lambda &  =\left\Vert u\right\Vert _{N}.
\end{align*}
Then,
\[
\nabla v\left(  x\right)  =\lambda\nabla u\left(  \lambda x\right)  .
\]
Hence
\[
\left\Vert \nabla v\right\Vert _{N}=\left\Vert \nabla u\right\Vert _{N}%
\leq1;\left\Vert v\right\Vert _{N}=1,
\]
and
\begin{align*}
&  \int_{%
\mathbb{R}
^{N}}\phi_{N}\left(  \alpha\left(  1-\frac{\beta}{N}\right)  \left\vert
v\left(  x\right)  \right\vert ^{\frac{N}{N-1}}\right)  \frac{dx}{\left\vert
x\right\vert ^{\beta}}\\
&  =\int_{%
\mathbb{R}
^{N}}\phi_{N}\left(  \alpha\left(  1-\frac{\beta}{N}\right)  \left\vert
u\left(  \lambda x\right)  \right\vert ^{\frac{N}{N-1}}\right)  \frac
{dx}{\left\vert x\right\vert ^{\beta}}\\
&  =\frac{1}{\lambda^{N-\beta}}\int_{%
\mathbb{R}
^{N}}\phi_{N}\left(  \alpha\left(  1-\frac{\beta}{N}\right)  \left\vert
u\left(  \lambda x\right)  \right\vert ^{\frac{N}{N-1}}\right)  \frac{d\left(
\lambda x\right)  }{\left\vert \lambda x\right\vert ^{\beta}}\\
&  =\frac{1}{\left\Vert u\right\Vert _{N}^{N-\beta}}\int_{%
\mathbb{R}
^{N}}\phi_{N}\left(  \alpha\left(  1-\frac{\beta}{N}\right)  \left\vert
u\right\vert ^{\frac{N}{N-1}}\right)  \frac{dx}{\left\vert x\right\vert
^{\beta}}.
\end{align*}

\end{proof}
\end{lemma}

By Lemma \ref{normalized1}, we can always assume $\left\Vert u\right\Vert
_{N}=1$ in the sharp subcritical Trudinger-Moser inequality.

\begin{lemma}
\label{consequence1}The sharp subcritical Moser-Trudinger inequality is a
consequence of the sharp critical Moser-Trudinger inequality. More precisely,
if $MT_{a,b}\left(  \beta\right)  $ is finite, then $AT\left(  \alpha
,\beta\right)  $ is finite. Moreover,%
\begin{equation}
AT\left(  \alpha,\beta\right)  \leq\left(  \frac{\left(  \frac{\alpha}%
{\alpha_{N}}\right)  ^{\frac{N-1}{N}b}}{1-\left(  \frac{\alpha}{\alpha_{N}%
}\right)  ^{\frac{N-1}{N}a}}\right)  ^{\frac{N-\beta}{b}}MT_{a,b}\left(
\beta\right)  . \label{2.1}%
\end{equation}
In particular,
\[
AT\left(  \alpha,\beta\right)  \leq\left(  \frac{\left(  \frac{\alpha}%
{\alpha_{N}}\right)  ^{N-1}}{1-\left(  \frac{\alpha}{\alpha_{N}}\right)
^{N-1}}\right)  ^{1-\frac{\beta}{N}}MT\left(  \beta\right)  .
\]

\begin{proof}
Let $u\in W^{1,N}\left(
\mathbb{R}
^{N}\right)  :~\left\Vert \nabla u\right\Vert _{N}\leq1;~\left\Vert
u\right\Vert _{N}=1$. Set%
\begin{align*}
v\left(  x\right)   &  =\left(  \frac{\alpha}{\alpha_{N}}\right)  ^{\frac
{N-1}{N}}u\left(  \lambda x\right) \\
\lambda &  =\left(  \frac{\left(  \frac{\alpha}{\alpha_{N}}\right)
^{\frac{N-1}{N}b}}{1-\left(  \frac{\alpha}{\alpha_{N}}\right)  ^{\frac{N-1}%
{N}a}}\right)  ^{1/b}.
\end{align*}
then
\begin{align*}
\left\Vert \nabla v\right\Vert _{N}^{a}  &  =\left(  \frac{\alpha}{\alpha_{N}%
}\right)  ^{\frac{N-1}{N}a}\left\Vert \nabla u\right\Vert _{N}^{a}\leq\left(
\frac{\alpha}{\alpha_{N}}\right)  ^{\frac{N-1}{N}a}\\
\left\Vert v\right\Vert _{N}^{b}  &  =\left(  \frac{\alpha}{\alpha_{N}%
}\right)  ^{\frac{N-1}{N}b}\frac{1}{\lambda^{b}}\left\Vert u\right\Vert
_{N}^{b}=1-\left(  \frac{\alpha}{\alpha_{N}}\right)  ^{\frac{N-1}{N}a}.
\end{align*}
Hence $\left\Vert \nabla v\right\Vert _{N}^{a}+\left\Vert v\right\Vert
_{N}^{b}\leq1.$ By the definition of $MT_{a,b}\left(  \beta\right)  $, we have%
\begin{align*}
&  \int_{%
\mathbb{R}
^{N}}\phi_{N}\left(  \alpha(1-\frac{\beta}{N})\left\vert u\right\vert
^{N/(N-1)}\right)  \frac{dx}{\left\vert x\right\vert ^{\beta}}\\
&  =\int_{%
\mathbb{R}
^{N}}\phi_{N}\left(  \alpha(1-\frac{\beta}{N})\left\vert u\left(  \lambda
x\right)  \right\vert ^{N/(N-1)}\right)  \frac{d\left(  \lambda x\right)
}{\left\vert \lambda x\right\vert ^{\beta}}\\
&  =\lambda^{N-\beta}\int_{%
\mathbb{R}
^{N}}\phi_{N}\left(  \alpha_{N}(1-\frac{\beta}{N})\left\vert v\right\vert
^{N/(N-1)}\right)  \frac{dx}{\left\vert x\right\vert ^{\beta}}\\
&  \leq\left(  \frac{\left(  \frac{\alpha}{\alpha_{N}}\right)  ^{\frac{N-1}%
{N}b}}{1-\left(  \frac{\alpha}{\alpha_{N}}\right)  ^{\frac{N-1}{N}a}}\right)
^{\frac{N-\beta}{b}}MT_{a,b}\left(  \beta\right)  .
\end{align*}

\end{proof}
\end{lemma}

\begin{lemma}
\label{normalized2}%
\[
ATA\left(  \alpha,\beta\right)  =\sup_{\left\Vert \Delta u\right\Vert
_{\frac{N}{2}}\leq1;\left\Vert u\right\Vert _{\frac{N}{2}}=1}\int_{%
\mathbb{R}
^{N}}\frac{\phi_{N,2}\left(  \alpha\left(  1-\frac{\beta}{N}\right)
\left\vert u\right\vert ^{\frac{N}{N-2}}\right)  }{\left\vert x\right\vert
^{\beta}}dx.
\]

\begin{proof}
Let $u\in W^{2,\frac{N}{2}}\left(
\mathbb{R}
^{N}\right)  :\left\Vert \Delta u\right\Vert _{\frac{N}{2}}\leq1$ and set%
\begin{align*}
v\left(  x\right)   &  =u\left(  \lambda x\right)  ;\\
\lambda &  =\left\Vert u\right\Vert _{\frac{N}{2}}^{\frac{1}{2}}%
\end{align*}
Then it is easy to check that%
\[
\Delta v\left(  x\right)  =\lambda^{2}\Delta u\left(  \lambda x\right)
\]
and
\begin{align*}
\left\Vert \Delta v\right\Vert _{\frac{N}{2}}  &  =\left\Vert \Delta
u\right\Vert _{\frac{N}{2}};\\
\left\Vert v\right\Vert _{\frac{N}{2}}^{\frac{N}{2}}  &  =\int_{%
\mathbb{R}
^{N}}\left\vert v\left(  x\right)  \right\vert ^{\frac{N}{2}}dx=\int_{%
\mathbb{R}
^{N}}\left\vert u\left(  \lambda x\right)  \right\vert ^{\frac{N}{2}}%
dx=\frac{1}{\lambda^{N}}\int_{%
\mathbb{R}
^{N}}\left\vert u\left(  x\right)  \right\vert ^{\frac{N}{2}}dx=1.
\end{align*}
Moreover%
\begin{align*}
\int_{%
\mathbb{R}
^{N}}\frac{\phi_{N,2}\left(  \alpha\left(  1-\frac{\beta}{N}\right)
\left\vert v\right\vert ^{\frac{N}{N-2}}\right)  }{\left\vert x\right\vert
^{\beta}}dx  &  =\int_{%
\mathbb{R}
^{N}}\frac{\phi_{N,2}\left(  \alpha\left(  1-\frac{\beta}{N}\right)
\left\vert u\left(  \lambda x\right)  \right\vert ^{\frac{N}{N-2}}\right)
}{\left\vert x\right\vert ^{\beta}}dx\\
&  =\frac{1}{\lambda^{N-\beta}}\int_{%
\mathbb{R}
^{N}}\frac{\phi_{N,2}\left(  \alpha\left(  1-\frac{\beta}{N}\right)
\left\vert u\left(  x\right)  \right\vert ^{\frac{N}{N-2}}\right)
}{\left\vert x\right\vert ^{\beta}}dx\\
&  =\frac{1}{\left\Vert u\right\Vert _{\frac{N}{2}}^{\frac{N}{2}\left(
1-\frac{\beta}{N}\right)  }}\int_{%
\mathbb{R}
^{N}}\frac{\phi_{N,2}\left(  \alpha\left(  1-\frac{\beta}{N}\right)
\left\vert u\right\vert ^{\frac{N}{N-2}}\right)  }{\left\vert x\right\vert
^{\beta}}dx.
\end{align*}

\end{proof}
\end{lemma}

\begin{lemma}
\label{consequence2}Assume $A_{a,b}\left(  \beta\right)  <\infty$, then
$ATA\left(  \alpha,\beta\right)  <\infty$. Moreover,
\begin{equation}
ATA\left(  \alpha,\beta\right)  \leq\left(  \frac{\left(  \frac{\alpha}%
{\beta\left(  N,2\right)  }\right)  ^{\frac{N-2}{N}b}}{1-\left(  \frac{\alpha
}{\beta\left(  N,2\right)  }\right)  ^{\frac{N-2}{N}a}}\right)  ^{\frac
{N-\beta}{2b}}A_{a,b}\left(  \beta\right)  . \label{2.2}%
\end{equation}
In particular, if $A\left(  \beta\right)  <\infty$, then
\[
ATA\left(  \alpha,\beta\right)  \leq\left(  \frac{\left(  \frac{\alpha}%
{\beta\left(  N,2\right)  }\right)  ^{\frac{N-2}{2}}}{1-\left(  \frac{\alpha
}{\beta\left(  N,2\right)  }\right)  ^{\frac{N-2}{2}}}\right)  ^{\frac
{N-\beta}{N}}A\left(  \beta\right)  .
\]

\begin{proof}
Let $u\in W^{2,\frac{N}{2}}\left(
\mathbb{R}
^{N}\right)  :\left\Vert \Delta u\right\Vert _{\frac{N}{2}}\leq1$ and
$\left\Vert u\right\Vert _{\frac{N}{2}}=1$. We define%
\begin{align*}
v\left(  x\right)   &  =\left(  \frac{\alpha}{\beta\left(  N,2\right)
}\right)  ^{\frac{N-2}{N}}u\left(  \lambda x\right) \\
\lambda &  =\left(  \frac{\left(  \frac{\alpha}{\beta\left(  N,2\right)
}\right)  ^{\frac{N-2}{N}b}}{1-\left(  \frac{\alpha}{\beta\left(  N,2\right)
}\right)  ^{\frac{N-2}{N}a}}\right)  ^{\frac{1}{2b}}.
\end{align*}
then
\begin{align*}
\left\Vert \Delta v\right\Vert _{\frac{N}{2}}  &  =\left(  \frac{\alpha}%
{\beta\left(  N,2\right)  }\right)  ^{\frac{N-2}{N}}\left\Vert \Delta
u\right\Vert _{\frac{N}{2}}\leq\left(  \frac{\alpha}{\beta\left(  N,2\right)
}\right)  ^{\frac{N-2}{N}}\\
\left\Vert v\right\Vert _{\frac{N}{2}}^{b}  &  =\left(  \frac{\alpha}%
{\beta\left(  N,2\right)  }\right)  ^{\frac{N-2}{N}b}\frac{1}{\lambda^{2b}%
}\left\Vert u\right\Vert _{\frac{N}{2}}^{b}=1-\left(  \frac{\alpha}%
{\beta\left(  N,2\right)  }\right)  ^{\frac{N-2}{N}a}.
\end{align*}
Hence $\left\Vert \Delta v\right\Vert _{\frac{N}{2}}^{a}+\left\Vert
v\right\Vert _{\frac{N}{2}}^{b}\leq1.$ By the definition of $A_{a,b}\left(
\beta\right)  $, we have%
\begin{align*}
&  \int_{%
\mathbb{R}
^{N}}\phi_{N,2}\left(  \alpha(1-\frac{\beta}{N})\left\vert u\right\vert
^{N/(N-2)}\right)  \frac{dx}{\left\vert x\right\vert ^{\beta}}\\
&  =\int_{%
\mathbb{R}
^{N}}\phi_{N,2}\left(  \alpha(1-\frac{\beta}{N})\left\vert u\left(  \lambda
x\right)  \right\vert ^{N/(N-2)}\right)  \frac{d\left(  \lambda x\right)
}{\left\vert \lambda x\right\vert ^{\beta}}\\
&  =\lambda^{N-\beta}\int_{%
\mathbb{R}
^{N}}\phi_{N}\left(  \alpha_{N}(1-\frac{\beta}{N})\left\vert v\right\vert
^{N/(N-2)}\right)  \frac{dx}{\left\vert x\right\vert ^{\beta}}\\
&  \leq\left(  \frac{\left(  \frac{\alpha}{\beta\left(  N,2\right)  }\right)
^{\frac{N-2}{N}b}}{1-\left(  \frac{\alpha}{\beta\left(  N,2\right)  }\right)
^{\frac{N-2}{N}a}}\right)  ^{\frac{N-\beta}{2b}}A_{a,b}\left(  \beta\right)  .
\end{align*}

\end{proof}
\end{lemma}

\section{Asymptotic behavior of the supremums in subcritical Trudinger-Moser
inequalities and relationship with the critical supremums}

In this section, we will prove the improved sharp subcritical Trudinger-Moser
inequality. In particular, we will establish the asymptotic behavior for the
supremum $AT(\alpha, \beta)$ for the subcritical Trudinger-Moser inequality
(Theorem 1.1). We would like to note here that we don't assume the critical
$MT\left(  \beta\right)  <\infty$ in the proof of Theorem \ref{improvedAT}.
Moreover, we also establish the relationship between the supremums $AT(\alpha,
\beta)$ and $MT(\beta)$ of the critical and subcritical Trudinger-Moser
inequalities (Theorem 1.2).

\begin{proof}
[Proof of Theorem \ref{improvedAT}]Suppose that $u\in C_{0}^{\infty}\left(
\mathbb{R}
^{N}\right)  \setminus\left\{  0\right\}  $,$~u\geq0$, $\left\Vert \nabla
u\right\Vert _{N}\leq1$ and $\left\Vert u\right\Vert _{N}=1$. Let%
\[
\Omega=\left\{  x:u\left(  x\right)  >\left(  1-\left(  \frac{\alpha}%
{\alpha_{N}}\right)  ^{N-1}\right)  ^{\frac{1}{N}}\right\}  .
\]
Then the volume of $\Omega$ can be estimated as follows:
\[
\left\vert \Omega\right\vert =%
{\displaystyle\int\limits_{\Omega}}
1dx\leq%
{\displaystyle\int\limits_{\Omega}}
\frac{u\left(  x\right)  ^{N}}{1-\left(  \frac{\alpha}{\alpha_{N}}\right)
^{N-1}}dx\leq\frac{1}{1-\left(  \frac{\alpha}{\alpha_{N}}\right)  ^{N-1}}.
\]
We have%
\begin{align*}
&
{\displaystyle\int\limits_{\mathbb{R}^{N}\setminus\Omega}}
\frac{\phi_{N}\left(  \alpha\left(  1-\frac{\beta}{N}\right)  \left\vert
u\right\vert ^{N/(N-1)}\right)  }{\left\vert x\right\vert ^{\beta}}dx\\
&  \leq%
{\displaystyle\int\limits_{\left\{  u\leq1\right\}  }}
\frac{\phi_{N}\left(  \alpha\left\vert u\right\vert ^{N/(N-1)}\right)
}{\left\vert x\right\vert ^{\beta}}dx\\
&  \leq e^{\alpha}%
{\displaystyle\int\limits_{\left\{  u\leq1\right\}  }}
\frac{u^{N}}{\left\vert x\right\vert ^{\beta}}dx\\
&  \leq e^{\alpha}%
{\displaystyle\int\limits_{\left\{  u\leq1;\left\vert x\right\vert
\geq1\right\}  }}
\frac{u^{N}}{\left\vert x\right\vert ^{\beta}}dx+e^{\alpha}%
{\displaystyle\int\limits_{\left\{  u\leq1;\left\vert x\right\vert <1\right\}
}}
\frac{u^{N}}{\left\vert x\right\vert ^{\beta}}dx\\
&  \leq\frac{C\left(  N,\beta\right)  }{\left(  1-\left(  \frac{\alpha}%
{\alpha_{N}}\right)  ^{N-1}\right)  ^{\left(  N-\beta\right)  /N}}.
\end{align*}
Now, consider%
\begin{align*}
I  &  =%
{\displaystyle\int\limits_{\Omega}}
\frac{\phi_{N}\left(  \alpha\left(  1-\frac{\beta}{N}\right)  \left\vert
u\right\vert ^{N/(N-1)}\right)  }{\left\vert x\right\vert ^{\beta}}dx\\
&  \leq%
{\displaystyle\int\limits_{\Omega}}
\frac{\exp\left(  \alpha\left(  1-\frac{\beta}{N}\right)  \left\vert
u\right\vert ^{N/(N-1)}\right)  }{\left\vert x\right\vert ^{\beta}}dx.
\end{align*}
On $\Omega$, we set%
\[
v\left(  x\right)  =u\left(  x\right)  -\left(  1-\left(  \frac{\alpha}%
{\alpha_{N}}\right)  ^{N-1}\right)  ^{\frac{1}{N}}.
\]
Then it is clear that $v\in W_{0}^{1,N}\left(  \Omega\right)  $ and
$\left\Vert \nabla v\right\Vert _{N}\leq1.$ Also, on $\Omega$, with
$\varepsilon=\frac{\alpha_{N}}{\alpha}-1:$%
\begin{align*}
\left\vert u\right\vert ^{N/(N-1)}  &  \leq\left(  \left\vert v\right\vert
+\left(  1-\left(  \frac{\alpha}{\alpha_{N}}\right)  ^{N-1}\right)  ^{\frac
{1}{N}}\right)  ^{N/(N-1)}\\
&  \leq{(1+\varepsilon)|v|^{N/(N-1)}}+(1-\frac{1}{(1+\varepsilon)^{N-1}%
})^{\frac{1}{1-N}}|\left(  1-\left(  \frac{\alpha}{\alpha_{N}}\right)
^{N-1}\right)  ^{\frac{1}{N}}|^{N/(N-1)}\\
&  =\frac{\alpha_{N}}{\alpha}{|v|^{N/(N-1)}+1.}%
\end{align*}
Hence, by Moser-Trudinger inequality on bounded domains:%
\begin{align*}
I  &  \leq%
{\displaystyle\int\limits_{\Omega}}
\frac{\exp\left(  \alpha\left(  1-\frac{\beta}{N}\right)  \left\vert
u\right\vert ^{N/(N-1)}\right)  }{\left\vert x\right\vert ^{\beta}}dx\\
&  \leq%
{\displaystyle\int\limits_{\Omega}}
\frac{\exp\left(  \alpha_{N}\left(  1-\frac{\beta}{N}\right)  {|v|^{N/(N-1)}%
+\alpha}\right)  }{\left\vert x\right\vert ^{\beta}}dx\\
&  \leq C\left(  N,\beta\right)  \left\vert \Omega\right\vert ^{1-\frac{\beta
}{N}}\\
&  \leq\frac{C\left(  N,\beta\right)  }{\left(  1-\left(  \frac{\alpha}%
{\alpha_{N}}\right)  ^{N-1}\right)  ^{\left(  N-\beta\right)  /N}}.
\end{align*}
In conclusion, we have%
\[
AT\left(  \alpha,\beta\right)  \leq\frac{C\left(  N,\beta\right)  }{\left(
1-\left(  \frac{\alpha}{\alpha_{N}}\right)  ^{N-1}\right)  ^{\left(
N-\beta\right)  /N}}.
\]
Next, we will show that $AT\left(  \alpha_{N},\beta\right)  =\infty.$ Indeed,
consider the following sequence:%
\[
u_{n}(x)=\left\{
\begin{array}
[c]{c}%
0\text{ if }\left\vert x\right\vert \geq1,\\
\left(  \frac{N-\beta}{\omega_{N-1}n}\right)  ^{1/N}\log\left(  \frac
{1}{\left\vert x\right\vert }\right)  \text{ if }e^{-\frac{n}{N-\beta}%
}<\left\vert x\right\vert <1\\
\left(  \frac{1}{\omega_{N-1}}\right)  ^{\frac{1}{N}}\left(  \frac{n}{N-\beta
}\right)  ^{\frac{N-1}{N}}\text{ if }0\leq\left\vert x\right\vert \leq
e^{-\frac{n}{N-\beta}}%
\end{array}
\right.  .
\]
Then we can see easily that%
\[
\left\Vert \nabla u_{n}\right\Vert _{N}=1;~\left\Vert u_{n}\right\Vert
_{N}=o_{n}(1).
\]
However%
\begin{align*}
&
{\displaystyle\int\limits_{\mathbb{R}^{N}}}
\frac{\phi_{N}\left(  \alpha_{N}\left(  1-\frac{\beta}{N}\right)  \left\vert
u_{n}\right\vert ^{N/(N-1)}\right)  }{\left\vert x\right\vert ^{\beta}}dx\\
&  \geq%
{\displaystyle\int\limits_{\left\{  0\leq\left\vert x\right\vert \leq
e^{-\frac{n}{N-\beta}}\right\}  }}
\frac{\phi_{N}\left(  n\right)  }{\left\vert x\right\vert ^{\beta}}dx\\
&  =\omega_{N-1}\phi_{N}\left(  n\right)
{\displaystyle\int\limits_{0}^{e^{-\frac{n}{N-\beta}}}}
r^{N-1-\beta}dr\\
&  =\frac{\omega_{N-1}\phi_{N}\left(  n\right)  }{e^{n}\left(  N-\beta\right)
}\rightarrow\frac{\omega_{N-1}}{N-\beta}\text{ as }n\rightarrow\infty.
\end{align*}

Now, it is clear that there exists a large constant $M_{1}$, such that when
$n\geq M_{1}$,
\begin{align*}
\Vert u_{n}\Vert_{N}^{N}  &  =\int_{0}^{e^{-\frac{n}{N-\beta}}}({\frac
{1}{\omega_{N-1}}})^{N/N}({\frac{n}{N-\beta}})^{\frac{N(N-1)}{N}}%
r^{N-1}dr+\int_{e^{-\frac{n}{N-\beta}}}^{1}({\frac{N-\beta}{\omega_{N-1}n}%
})^{N/N}(\log{({\frac{1}{r}})})^{N}r^{N-1}dr\\
&  \approx n^{N-1}\int_{0}^{e^{-\frac{n}{N-\beta}}}r^{N-1}dr+{\frac{1}{n}}%
\int_{0}^{\frac{n}{N-\beta}}y^{N}e^{-Ny}dy\\
&  \approx n^{N}e^{-{\frac{nN}{N-\beta}}}+{\frac{1}{n}}\approx{\frac{1}{n}}%
\end{align*}
So
\[
\Vert u_{n}\Vert_{N}^{N-\beta}\approx\frac{1}{n^{N-\beta}}\quad\text{when}%
\quad n\geq M_{1}.
\]
Now we consider the following integral
\begin{align*}
&  \int_{\mathbb{R}^{N}}{\frac{\phi_{N}(\alpha(1-\beta/N)|u_{n}|^{\frac
{N}{N-1}})}{|x|^{\beta}}}dx\\
&  \gtrsim\int_{0}^{e^{-\frac{n}{N-\beta}}}\phi_{N}{\left(  \alpha
(1-\beta/N)({\frac{1}{\omega_{N-1}}})^{\frac{1}{N-1}}({\frac{n}{N-\beta}%
})\right)  }r^{N-1-\beta}dr\\
&  \gtrsim\int_{0}^{e^{-\frac{n}{N-\beta}}}\phi_{N}{{\left(  {\frac{\alpha
}{\alpha_{N}}}n\right)  }}r^{N-1-\beta}dr\gtrsim\phi_{N}{{\left(
{\frac{\alpha}{\alpha_{N}}}n\right)  e^{-n}}}%
\end{align*}
We note that there exists a large constant $M_{2}$ independent of $\alpha$
such that for $n\geq M_{2}$
\[
\phi_{N}{\left(  {\frac{\alpha}{\alpha_{N}}}n\right)  }\approx e^{({\frac
{\alpha}{\alpha_{N}}})n}%
\]
as long as ${\frac{\alpha}{\alpha_{N}}}\geq{\frac{1}{2}}$.\newline Now we
have
\begin{align*}
&  \int_{\mathbb{R}^{N}}{\frac{\phi_{N}(\alpha(1-\beta/N)|u|^{\frac{N}{N-1}}%
)}{|x|^{\beta}}}dx\\
&  \gtrsim e^{\left(  {\frac{\alpha}{\alpha_{N}}}n\right)  }e^{-n}%
=e^{-(1-{\frac{\alpha}{\alpha_{N}}})n}%
\end{align*}
Now for $\alpha$ that is close enough to $\alpha_{N}$ we can pick $n$ such
that $1\leq(1-{\frac{\alpha}{\alpha_{N}}})n\leq2$, i.e
\[
\alpha\approx(1-{\frac{1}{n}})\alpha_{N}\geq\left(  1-{\frac{1}{\max
{(M_{1},M_{2})}}}\right)  \alpha_{N}%
\]
or
\[
\max{(M_{1},M_{2})}\leq n\approx\frac{1}{1-{\frac{\alpha}{\alpha_{N}}}},
\]
Then
\begin{align}
&  \frac{1}{\Vert u_{n}\Vert_{N}^{N-\beta}}\int_{\mathbb{R}^{N}}{\frac
{\Phi_{N}(\alpha(1-\beta/N)|u_{n}|^{\frac{N}{N-1}})}{|x|^{\beta}}%
}dx\nonumber\\
&  \gtrsim n^{N-\beta}e^{-2}\nonumber\\
&  \approx\left(  \frac{1}{1-{\frac{\alpha}{\alpha_{N}}}}\right)  ^{N-\beta}
\label{lowerbound}%
\end{align}
And note that when $\alpha$ is close enough to $\alpha_{N}$, we have
\[
\frac{1-({\frac{\alpha}{\alpha_{N}}})^{N-1}}{1-{\frac{\alpha}{\alpha_{N}}}%
}\approx1,
\]
which implies
\[
AT\left(  \alpha,\beta\right)  \geq\frac{c\left(  N,\beta\right)  }{\left(
1-\left(  \frac{\alpha}{\alpha_{N}}\right)  ^{N-1}\right)  ^{\left(
N-\beta\right)  /N}}%
\]
when $\alpha$ is close enough to $\alpha_{N}.$
\end{proof}

Now, we will provide a proof of the sharp critical Trudinger-Moser inequality,
namely Theorem \ref{MT}, using the above improved sharp subcritical
Trudinger-Moser inequality (\ref{1.3.1}). This suggests a new approach to and
another look at the study of the sharp Trudinger-Moser inequality:

\begin{proof}
[Proof of Theorem \ref{MT}]First assume that $b\leq N.$ Let $u\in
W^{1,N}\left(
\mathbb{R}
^{N}\right)  \setminus\left\{  0\right\}  :\left\Vert \nabla u\right\Vert
_{N}^{a}+\left\Vert u\right\Vert _{N}^{b}\leq1.$ Assume that%
\[
\left\Vert \nabla u\right\Vert _{N}=\theta\in\left(  0,1\right)  ;~\left\Vert
u\right\Vert _{N}^{b}\leq1-\theta^{a}.
\]
If $\frac{1}{2}<\theta<1$, then we set%
\begin{align*}
v\left(  x\right)   &  =\frac{u\left(  \lambda x\right)  }{\theta}\\
\lambda &  =\frac{\left(  1-\theta^{a}\right)  ^{\frac{1}{b}}}{\theta}>0.
\end{align*}
Hence%
\begin{align*}
\left\Vert \nabla v\right\Vert _{N}  &  =\frac{\left\Vert \nabla u\right\Vert
_{N}}{\theta}=1;\\
\left\Vert v\right\Vert _{N}^{N}  &  =\int_{%
\mathbb{R}
^{N}}\left\vert v\right\vert ^{N}dx=\frac{1}{\theta^{N}}\int_{%
\mathbb{R}
^{N}}\left\vert u\left(  \lambda x\right)  \right\vert ^{N}dx=\frac{1}%
{\theta^{N}\lambda^{N}}\left\Vert u\right\Vert _{N}^{N}\leq\frac{\left(
1-\theta^{a}\right)  ^{\frac{N}{b}}}{\theta^{N}\lambda^{N}}=1.
\end{align*}
By Theorem \ref{improvedAT}, we get%
\begin{align*}
&  \int_{%
\mathbb{R}
^{N}}\frac{\phi_{N}\left(  \alpha_{N}\left(  1-\frac{\beta}{N}\right)
\left\vert u\right\vert ^{\frac{N}{N-1}}\right)  }{\left\vert x\right\vert
^{\beta}}dx=\int_{%
\mathbb{R}
^{N}}\frac{\phi_{N}\left(  \alpha_{N}\left(  1-\frac{\beta}{N}\right)
\left\vert u\left(  \lambda x\right)  \right\vert ^{\frac{N}{N-1}}\right)
}{\left\vert \lambda x\right\vert ^{\beta}}d\left(  \lambda x\right) \\
&  \leq\lambda^{N-\beta}\int_{%
\mathbb{R}
^{N}}\frac{\phi_{N}\left(  \theta^{\frac{N}{N-1}}\alpha_{N}(1-\frac{\beta}%
{N})\left\vert v\right\vert ^{N/(N-1)}\right)  }{\left\vert x\right\vert
^{\beta}}dx\\
&  \leq\lambda^{N-\beta}AT\left(  \theta^{\frac{N}{N-1}}\alpha_{N}%
,\beta\right)  \leq\left(  \frac{\left(  1-\theta^{a}\right)  ^{\frac{N}{b}}%
}{\theta^{N}}\right)  ^{1-\frac{\beta}{N}}\frac{C\left(  N,\beta\right)
}{\left(  1-\left(  \frac{\theta^{\frac{N}{N-1}}\alpha_{N}}{\alpha_{N}%
}\right)  ^{N-1}\right)  ^{1-\frac{\beta}{N}}}\\
&  \leq\frac{\left(  \left(  1-\theta^{a}\right)  ^{\frac{N}{b}}\right)
^{1-\frac{\beta}{N}}}{\left(  1-\theta^{N}\right)  ^{1-\frac{\beta}{N}}%
}C\left(  N,\beta\right)  \leq C\left(  N,\beta,a,b\right)  \text{ since
}b\leq N\text{.}%
\end{align*}
If $0<\theta\leq\frac{1}{2}$, then with
\[
v\left(  x\right)  =2u\left(  2x\right)  ,
\]
we have%
\begin{align*}
\left\Vert \nabla v\right\Vert _{N}  &  =2\left\Vert \nabla u\right\Vert
_{N}\leq1\\
\left\Vert v\right\Vert _{N}  &  \leq1.
\end{align*}
By Theorem \ref{improvedAT}:%
\begin{align*}
&  \int_{%
\mathbb{R}
^{N}}\frac{\phi_{N}\left(  \alpha_{N}\left(  1-\frac{\beta}{N}\right)
\left\vert u\right\vert ^{\frac{N}{N-1}}\right)  }{\left\vert x\right\vert
^{\beta}}dx\leq2^{N}\int_{%
\mathbb{R}
^{N}}\frac{\phi_{N}\left(  \frac{\alpha_{N}(1-\frac{\beta}{N})}{2^{\frac
{N}{N-1}}}\left\vert v\right\vert ^{N/(N-1)}\right)  }{\left\vert x\right\vert
^{\beta}}dx\\
&  \leq C\left(  N,\beta\right)  .
\end{align*}
Next, we will verify that the constant $\alpha_{N}(1-\frac{\beta}{N})$ is our
best possible. Indeed, we choose the sequence $\left\{  u_{k}\right\}  $ as
follows
\begin{equation}
u_{n}(x)=\left\{
\begin{array}
[c]{c}%
0\text{ if }\left\vert x\right\vert \geq1,\\
\left(  \frac{N-\beta}{\omega_{N-1}n}\right)  ^{1/N}\log\left(  \frac
{1}{\left\vert x\right\vert }\right)  \text{ if }e^{-\frac{n}{N-\beta}%
}<\left\vert x\right\vert <1\\
\left(  \frac{1}{\omega_{N-1}}\right)  ^{\frac{1}{N}}\left(  \frac{n}{N-\beta
}\right)  ^{\frac{N-1}{N}}\text{ if }0\leq\left\vert x\right\vert \leq
e^{-\frac{n}{N-\beta}}%
\end{array}
\right.  . \label{3.1}%
\end{equation}
Then,%
\[
\left\Vert \nabla u_{n}\right\Vert _{N}=1;~\left\Vert u_{n}\right\Vert
_{N}=O(\frac{1}{n^{\frac{1}{N}}}).
\]
Set
\begin{align*}
~w_{n}(x)  &  =\lambda_{n}u_{n}\left(  x\right)  \text{ where }\lambda_{n}%
\in\left(  0,1\right)  \text{ is a solution of }\lambda_{n}^{a}+\lambda
_{n}^{b}\left\Vert u_{n}\right\Vert _{N}^{b}=1.\\
\lambda_{n}  &  =1-O\left(  \frac{1}{n^{\frac{b}{aN}}}\right)  \rightarrow
_{k\rightarrow\infty}1.
\end{align*}
Then
\[
\left\Vert \nabla w_{n}\right\Vert _{N}^{a}+\left\Vert w_{n}\right\Vert
_{N}^{b}=1.
\]
Also, for $\alpha>\alpha_{N}:$%
\begin{align*}
&  \int_{%
\mathbb{R}
^{N}}\frac{\phi_{N}\left(  \alpha\left(  1-\frac{\beta}{N}\right)  \left\vert
w_{n}\right\vert ^{\frac{N}{N-1}}\right)  }{\left\vert x\right\vert ^{\beta}%
}dx\\
&  \geq%
{\displaystyle\int\limits_{\left\{  0\leq\left\vert x\right\vert \leq
e^{-\frac{n}{N-\beta}}\right\}  }}
\frac{\exp\left(  \alpha\left(  1-\frac{\beta}{N}\right)  \left\vert
w_{n}\right\vert ^{\frac{N}{N-1}}\right)  -%
{\displaystyle\sum\limits_{j=0}^{N-2}}
\frac{\left[  \alpha\left(  1-\frac{\beta}{N}\right)  \right]  ^{j}}%
{j!}\left\vert w_{n}\right\vert ^{\frac{N}{N-1}j}}{\left\vert x\right\vert
^{\beta}}dx\\
&  \geq\left[  \exp\left(  \frac{\alpha n\left(  1-O\left(  \frac{1}%
{n^{\frac{b}{a\left(  N-1\right)  }}}\right)  \right)  }{\alpha_{N}}\right)
-O\left(  k^{N-1}\right)  \right]  \frac{\omega_{N-1}\exp\left(  -n\right)
}{N-\beta}\\
&  \rightarrow\infty\text{ as }n\rightarrow\infty\text{.}%
\end{align*}
Now, we will show that%
\[
MT_{a,b}\left(  \beta\right)  =\sup_{\alpha\in\left(  0,\alpha_{N}\right)
}\left(  \frac{1-\left(  \frac{\alpha}{\alpha_{N}}\right)  ^{\frac{N-1}{N}a}%
}{\left(  \frac{\alpha}{\alpha_{N}}\right)  ^{\frac{N-1}{N}b}}\right)
^{\frac{N-\beta}{b}}AT\left(  \alpha,\beta\right)
\]
when $MT_{a,b}\left(  \beta\right)  <\infty$. Indeed, by (\ref{2.1}), we have%
\[
\sup_{\alpha\in\left(  0,\alpha_{N}\right)  }\left(  \frac{1-\left(
\frac{\alpha}{\alpha_{N}}\right)  ^{\frac{N-1}{N}a}}{\left(  \frac{\alpha
}{\alpha_{N}}\right)  ^{\frac{N-1}{N}b}}\right)  ^{\frac{N-\beta}{b}}AT\left(
\alpha,\beta\right)  \leq MT_{a,b}\left(  \beta\right)  .
\]
Now, let $\left(  u_{n}\right)  $ be the maximizing sequence of $MT_{a,b}%
\left(  \beta\right)  $, i.e., $u_{n}\in W^{1,N}\left(
\mathbb{R}
^{N}\right)  \setminus\left\{  0\right\}  :\left\Vert \nabla u_{n}\right\Vert
_{N}^{a}+\left\Vert u_{n}\right\Vert _{N}^{b}\leq1$ and
\[
\int_{%
\mathbb{R}
^{N}}\phi_{N}\left(  \alpha_{N}\left(  1-\frac{\beta}{N}\right)  \left\vert
u_{n}\right\vert ^{\frac{N}{N-1}}\right)  \frac{dx}{\left\vert x\right\vert
^{\beta}}\rightarrow_{n\rightarrow\infty}MT_{a,b}\left(  \beta\right)  .
\]
We define%
\begin{align*}
v_{n}\left(  x\right)   &  =\frac{u\left(  \lambda_{n}x\right)  }{\left\Vert
\nabla u_{n}\right\Vert _{N}}\\
\lambda_{n}  &  =\left(  \frac{1-\left\Vert \nabla u_{n}\right\Vert _{N}^{a}%
}{\left\Vert \nabla u_{n}\right\Vert _{N}^{b}}\right)  ^{1/b}>0.
\end{align*}
Hence%
\[
\left\Vert \nabla v_{n}\right\Vert _{N}=1\text{ and }\left\Vert v_{n}%
\right\Vert _{N}\leq1.
\]
Also,
\begin{align*}
&  \int_{%
\mathbb{R}
^{N}}\phi_{N}\left(  \alpha_{N}\left(  1-\frac{\beta}{N}\right)  \left\vert
u_{n}\right\vert ^{\frac{N}{N-1}}\right)  \frac{dx}{\left\vert x\right\vert
^{\beta}}\\
&  =\lambda_{n}^{N-\beta}\int_{%
\mathbb{R}
^{N}}\frac{\phi_{N}\left(  \left\Vert \nabla u_{n}\right\Vert _{N}^{\frac
{N}{N-1}}\alpha_{N}(1-\frac{\beta}{N})\left\vert v_{n}\right\vert
^{N/(N-1)}\right)  }{\left\vert x\right\vert ^{\beta}}dx\\
&  \leq\lambda_{n}^{N-\beta}AT\left(  \left\Vert \nabla u_{n}\right\Vert
_{N}^{\frac{N}{N-1}}\alpha_{N},\beta\right)  \leq\sup_{\alpha\in\left(
0,\alpha_{N}\right)  }\left(  \frac{1-\left(  \frac{\alpha}{\alpha_{N}%
}\right)  ^{\frac{N-1}{N}a}}{\left(  \frac{\alpha}{\alpha_{N}}\right)
^{\frac{N-1}{N}b}}\right)  ^{\frac{N-\beta}{b}}AT\left(  \alpha,\beta\right)
.
\end{align*}
Hence, we receive
\[
MT_{a,b}\left(  \beta\right)  =\sup_{\alpha\in\left(  0,\alpha_{N}\right)
}\left(  \frac{1-\left(  \frac{\alpha}{\alpha_{N}}\right)  ^{\frac{N-1}{N}a}%
}{\left(  \frac{\alpha}{\alpha_{N}}\right)  ^{\frac{N-1}{N}b}}\right)
^{\frac{N-\beta}{b}}AT\left(  \alpha,\beta\right)
\]
when $MT_{a,b}\left(  \beta\right)  <\infty.$

Now, if there exists some $b>N$ such that $MT_{a,b}\left(  \beta\right)
<\infty$. Then we have
\[
\overline{\lim}_{\alpha\rightarrow\alpha_{N}^{-}}\left(  1-\left(
\frac{\alpha}{\alpha_{N}}\right)  ^{\frac{N-1}{N}a}\right)  ^{\frac{N-\beta
}{b}}AT\left(  \alpha,\beta\right)  <\infty\text{.}%
\]
Also, since $MT\left(  \beta\right)  <\infty:$%
\[
\overline{\lim}_{\alpha\rightarrow\alpha_{N}^{-}}\left(  1-\left(
\frac{\alpha}{\alpha_{N}}\right)  ^{N-1}\right)  ^{\frac{N-\beta}{N}}AT\left(
\alpha,\beta\right)  <\infty\text{.}%
\]
By Theorem \ref{improvedAT}, we can show that
\begin{equation}
\overline{\lim}_{\alpha\rightarrow\alpha_{N}^{-}}\left(  1-\left(
\frac{\alpha}{\alpha_{N}}\right)  ^{N-1}\right)  ^{\frac{N-\beta}{N}}AT\left(
\alpha,\beta\right)  >0. \label{3.2}%
\end{equation}
Hence%
\[
\underline{\lim}_{\alpha\rightarrow\alpha_{N}^{-}}\frac{\left(  1-\left(
\frac{\alpha}{\alpha_{N}}\right)  ^{\frac{N-1}{N}a}\right)  ^{\frac{N-\beta
}{b}}}{\left(  1-\left(  \frac{\alpha}{\alpha_{N}}\right)  ^{N-1}\right)
^{\frac{N-\beta}{N}}}<\infty
\]
which is impossible since $b>N$. The proof is now completed.
\end{proof}

\section{Asymptotic behavior of subcritical Adams inequalities and
relationship with the critical ones}

\subsection{Sharp Adams inequalities on $W^{2,\frac{N}{2}}\left(
\mathbb{R}
^{N}\right)  $}

In this subsection, we establish the asymptotic behavior of the supremums in
the subcritical Adams inequalities, namely Theorem \ref{ATA}. Again, it is
worthy noticing that no version of Theorem \ref{A} is assumed in order to
prove Theorem \ref{ATA}. Moreover, we also establish the relationship between
the supremums for the critical and subcritical Adams inequalities (Theorem 1.4).

\begin{proof}
[Proof of Theorem \ref{ATA}]Let $u\in C_{0}^{\infty}\left(
\mathbb{R}
^{N}\right)  {\setminus\{0\}}$, $u\geq{0}$, $\left\Vert \Delta u\right\Vert
_{\frac{N}{2}}\leq1$ and $\left\Vert u\right\Vert _{\frac{N}{2}}=1.$ Set
\[
\Omega(u)=\left\{  x\in%
\mathbb{R}
^{n}:u(x)>\left[  1-\left(  \frac{\alpha}{\beta\left(  N,2\right)  }\right)
^{\frac{N-2}{2}}\right]  ^{\frac{2}{N}}\right\}  .
\]
Since $u\in C_{0}^{\infty}\left(
\mathbb{R}
^{n}\right)  $, we have that $\Omega(u)$ is a bounded set. Moreover, we have
\[
\left\vert {\Omega(u)}\right\vert \leq{\int_{\Omega(u)}}\frac{{|u|^{\frac
{N}{2}}}}{1-\left(  \frac{\alpha}{\beta\left(  N,2\right)  }\right)
^{\frac{N-2}{2}}}dx\leq\frac{1}{1-\left(  \frac{\alpha}{\beta\left(
N,2\right)  }\right)  ^{\frac{N-2}{2}}}.
\]
Now, consider%
\begin{align*}
I &  =%
{\displaystyle\int\limits_{{\Omega(u)}}}
\frac{\phi_{N,2}\left(  \alpha\left(  1-\frac{\beta}{N}\right)  \left\vert
u\right\vert ^{N/(N-2)}\right)  }{\left\vert x\right\vert ^{\beta}}dx\\
&  \leq%
{\displaystyle\int\limits_{{\Omega(u)}}}
\frac{\exp\left(  \alpha\left(  1-\frac{\beta}{N}\right)  \left\vert
u\right\vert ^{N/(N-2)}\right)  }{\left\vert x\right\vert ^{\beta}}dx.
\end{align*}
On ${\Omega(u)}$, we set%
\[
v\left(  x\right)  =u\left(  x\right)  -\left[  1-\left(  \frac{\alpha}%
{\beta\left(  N,2\right)  }\right)  ^{\frac{N-2}{2}}\right]  ^{\frac{2}{N}}.
\]
Then it is clear that $v\in W_{N}^{2,{\frac{N}{2}}}\left(  {\Omega(u)}\right)
$ and $\left\Vert \Delta v\right\Vert _{\frac{N}{2}}\leq1.$ Also, on
${\Omega(u)}$, with $\varepsilon=\frac{\beta\left(  N,2\right)  }{\alpha}-1:$%
\begin{align*}
\left\vert u\right\vert ^{N/(N-2)} &  \leq\left(  \left\vert v\right\vert
+\left(  1-\left(  \frac{\alpha}{\beta\left(  N,2\right)  }\right)
^{\frac{N-2}{2}}\right)  ^{\frac{2}{N}}\right)  ^{N/(N-2)}\\
&  \leq{(1+\varepsilon)|v|^{N/(N-2)}}+(1-\frac{1}{(1+\varepsilon)^{\frac
{N-2}{2}}})^{\frac{2}{2-N}}|\left(  1-\left(  \frac{\alpha}{\beta\left(
N,2\right)  }\right)  ^{\frac{N-2}{2}}\right)  ^{\frac{2}{N}}|^{N/(N-2)}\\
&  =\frac{\beta\left(  N,2\right)  }{\alpha}{|v|^{N/(N-2)}+1.}%
\end{align*}
Hence, by Adams inequality on bounded domains (Theorem D):%
\begin{align*}
I &  \leq%
{\displaystyle\int\limits_{{\Omega(u)}}}
\frac{\exp\left(  \alpha\left(  1-\frac{\beta}{N}\right)  \left\vert
u\right\vert ^{N/(N-2)}\right)  }{\left\vert x\right\vert ^{\beta}}dx\\
&  \leq%
{\displaystyle\int\limits_{{\Omega(u)}}}
\frac{\exp\left(  \beta\left(  N,2\right)  \left(  1-\frac{\beta}{N}\right)
{|v|^{N/(N-2)}+\alpha}\right)  }{\left\vert x\right\vert ^{\beta}}dx\\
&  \leq C\left(  N,\beta\right)  \left\vert {\Omega(u)}\right\vert
^{1-\frac{\beta}{N}}\\
&  \leq C\left(  N,\beta\right)  \left(  \frac{1}{1-\left(  \frac{\alpha
}{\beta\left(  N,2\right)  }\right)  ^{\frac{N-2}{2}}}\right)  ^{1-\frac
{\beta}{N}}.
\end{align*}

We also have the following estimate:%
\begin{align*}
&
{\displaystyle\int\limits_{\mathbb{R}^{N}\setminus{\Omega(u)}}}
\phi_{N,2}\left(  \alpha(1-\frac{\beta}{N})\left\vert u\right\vert
^{N/(N-2)}\right)  \frac{dx}{\left\vert x\right\vert ^{\beta}}\\
&  \leq%
{\displaystyle\int\limits_{\left\{  u\leq1\right\}  }}
\phi_{N,2}\left(  \alpha(1-\frac{\beta}{N})\left\vert u\right\vert
^{N/(N-2)}\right)  \frac{dx}{\left\vert x\right\vert ^{\beta}}\\
&  \leq C\left(  N\right)
{\displaystyle\int\limits_{\left\{  u\leq1\right\}  }}
\frac{\left\vert u\right\vert ^{\frac{N}{2}}}{\left\vert x\right\vert ^{\beta
}}dx\\
&  \leq C\left(  N\right)  \left(
{\displaystyle\int\limits_{\left\{  u\leq1;\left\vert x\right\vert
\geq1\right\}  }}
\frac{\left\vert u\right\vert ^{\frac{N}{2}}}{\left\vert x\right\vert ^{\beta
}}dx+%
{\displaystyle\int\limits_{\left\{  u\leq1;\left\vert x\right\vert <1\right\}
}}
\frac{\left\vert u\right\vert ^{\frac{N}{2}}}{\left\vert x\right\vert ^{\beta
}}dx\right)  \\
&  \leq C\left(  N,\beta\right)  .
\end{align*}
In conclusion, we have%
\[
ATA\left(  \alpha,\beta\right)  \leq\frac{C\left(  N,\beta\right)  }{\left[
1-\left(  \frac{\alpha}{\beta\left(  N,2\right)  }\right)  ^{\frac{N-2}{2}%
}\right]  ^{1-\frac{\beta}{N}}}.
\]
We now show that $ATA\left(  \beta\left(  N,2\right)  ,\beta\right)  =\infty.$
Indeed, let $\psi\in C^{\infty}\left(  \left[  0,1\right]  \right)  $ be such
that
\[
\psi\left(  0\right)  =\psi^{\prime}\left(  0\right)  =0;\ \psi\left(
1\right)  =\psi^{\prime}\left(  1\right)  =1.
\]
For $0<\varepsilon<\frac{1}{2}$ we set%
\[
H\left(  t\right)  =\left\{
\begin{array}
[c]{c}%
\varepsilon\psi\left(  \frac{t}{\varepsilon}\right)  \text{
\ \ \ \ \ \ \ \ \ \ \ \ \ \ \ \ \ \ \ \ }0<t\leq\varepsilon\\
t\text{ ~\ \ \ \ \ \ \ \ \ \ \ \ \ \ \ \ \ \ \ \ \ \ }\varepsilon
<t\leq1-\varepsilon\\
1-\varepsilon\psi\left(  \frac{1-t}{\varepsilon}\right)
~\ \ \ \ \ 1-\varepsilon<t\leq1\\
0\text{
\ \ \ \ \ \ \ \ \ \ \ \ \ \ \ \ \ \ \ \ \ \ \ \ \ \ \ \ \ \ \ \ \ \ \ \ }1<t
\end{array}
\right.
\]
and consider Adams' test functions%
\[
\psi_{r}\left(  \left\vert x\right\vert \right)  =H\left(  \frac{\log\frac
{1}{\left\vert x\right\vert }}{\log\frac{1}{r}}\right)  .
\]
By construction, $\psi_{r}\in W^{2,\frac{N}{2}}\left(
\mathbb{R}
^{N}\right)  $ and $\psi_{r}\left(  \left\vert x\right\vert \right)  =1$ for
$x\in B_{r}.$ Moreover, by \cite{A}:%
\begin{align*}
\left\Vert \Delta\psi_{r}\right\Vert _{\frac{N}{2}}^{\frac{N}{2}} &
\leq\omega_{N-1}a\left(  N,2\right)  ^{\frac{N}{2}}\log\left(  \frac{1}%
{r}\right)  ^{1-\frac{N}{2}}A_{r};\\
\left\Vert \psi_{r}\right\Vert _{\frac{N}{2}}^{\frac{N}{2}} &  =o\left(
\left(  \frac{1}{\log\left(  \frac{1}{r}\right)  }\right)  ^{\frac{N-2}{2}%
}\right)  \\
a\left(  N,2\right)   &  =\frac{\beta\left(  N,2\right)  ^{\frac{N-2}{N}}%
}{N\sigma_{N}^{\frac{2}{N}}};\\
A_{r} &  =A_{r}\left(  N,2\right)  =\left[  1+2\varepsilon\left(  \left\Vert
\psi^{\prime}\right\Vert _{\infty}+O\left(  \frac{1}{\log\left(  \frac{1}%
{r}\right)  }\right)  \right)  ^{\frac{N}{2}}\right]  ;
\end{align*}
Now, we set%
\[
u_{r}\left(  \left\vert x\right\vert \right)  =\left(  \log\left(  \frac{1}%
{r}\right)  \right)  ^{\frac{N-2}{N}}\psi_{r}\left(  \left\vert x\right\vert
\right)  .
\]
Then%
\begin{align*}
u_{r}\left(  \left\vert x\right\vert \right)   &  =\left(  \log\left(
\frac{1}{r}\right)  \right)  ^{\frac{N-2}{N}}\text{ for }x\in B_{r}\\
\left\Vert \Delta u_{r}\right\Vert _{\frac{N}{2}}^{\frac{N}{2}} &  \leq
\omega_{N-1}a\left(  N,2\right)  ^{\frac{N}{2}}A_{r}\text{ and }\\
\left\Vert \Delta u_{r}\right\Vert _{\frac{N}{2}}^{\frac{N}{N-2}} &  \leq
\frac{\beta\left(  N,2\right)  }{N}A_{r}^{\frac{2}{N-2}}.
\end{align*}
Now,
\begin{align*}
ATA\left(  \beta\left(  N,2\right)  ,\beta\right)   &  \geq\lim_{r\rightarrow
0^{+}}\frac{1}{\left\Vert \frac{u_{r}}{\left\Vert \Delta u_{r}\right\Vert
_{\frac{N}{2}}}\right\Vert _{\frac{N}{2}}^{\frac{N}{2}\left(  1-\frac{\beta
}{N}\right)  }}%
{\displaystyle\int\limits_{B_{r}}}
\phi_{N,2}\left(  \beta\left(  N,2\right)  (1-\frac{\beta}{N})\left\vert
\frac{u_{r}}{\left\Vert \Delta u_{r}\right\Vert _{\frac{N}{2}}}\right\vert
^{N/(N-2)}\right)  \frac{dx}{\left\vert x\right\vert ^{\beta}}\\
&  \geq\lim_{r\rightarrow0^{+}}\frac{\left\Vert \Delta u_{r}\right\Vert
_{\frac{N}{2}}^{\frac{N}{2}\left(  1-\frac{\beta}{N}\right)  }}{\left\Vert
u_{r}\right\Vert _{\frac{N}{2}}^{\frac{N}{2}\left(  1-\frac{\beta}{N}\right)
}}%
{\displaystyle\int\limits_{B_{r}}}
\phi_{N,2}\left(  \frac{\beta\left(  N,2\right)  (1-\frac{\beta}{N}%
)\log\left(  \frac{1}{r}\right)  }{\left\Vert \Delta u_{r}\right\Vert
_{\frac{N}{2}}^{\frac{N}{N-2}}}\right)  \frac{dx}{\left\vert x\right\vert
^{\beta}}\\
&  \geq\lim_{r\rightarrow0^{+}}\frac{\left\Vert \Delta u_{r}\right\Vert
_{\frac{N}{2}}^{\frac{N}{2}\left(  1-\frac{\beta}{N}\right)  }}{\left\Vert
u_{r}\right\Vert _{\frac{N}{2}}^{\frac{N}{2}\left(  1-\frac{\beta}{N}\right)
}}\omega_{N-1}\frac{r^{N-\beta}}{N-\beta}\phi_{N,2}\left(  \frac{\left(
N-\beta\right)  \log\left(  \frac{1}{r}\right)  }{\left[  1+2\varepsilon
\left(  \left\Vert \psi^{\prime}\right\Vert _{\infty}+O\left(  \frac{1}%
{\log\left(  \frac{1}{r}\right)  }\right)  \right)  ^{\frac{N}{2}}\right]
^{\frac{2}{N-2}}}\right)  \\
&  \rightarrow\infty\text{ as }r\rightarrow0^{+}\text{.}%
\end{align*}

Now, consider the following sequence%
\[
u_{k}\left(  x\right)  =\left\{
\begin{array}
[c]{c}%
\left[  \frac{1}{\beta\left(  N,2\right)  }\ln k\right]  ^{1-\frac{2}{N}%
}-\frac{\left\vert x\right\vert ^{2}}{\left(  \frac{\ln k}{k}\right)
^{\frac{2}{N}}}+\frac{1}{\left(  \ln k\right)  ^{\frac{2}{N}}}\text{ if }%
0\leq\left\vert x\right\vert \leq\left(  \frac{1}{k}\right)  ^{\frac{1}{N}}\\
N\beta\left(  N,2\right)  ^{\frac{2}{N}-1}\left(  \ln k\right)  ^{-\frac{2}%
{N}}\ln\frac{1}{\left\vert x\right\vert }\text{
\ \ \ \ \ \ \ \ \ \ \ \ \ \ \ \ \ if }\left(  \frac{1}{k}\right)  ^{\frac
{1}{N}}\leq\left\vert x\right\vert \leq1.\\
0\text{
\ \ \ \ \ \ \ \ \ \ \ \ \ \ \ \ \ \ \ \ \ \ \ \ \ \ \ \ \ \ \ \ \ \ \ \ \ \ \ \ if
}\left\vert x\right\vert >1.
\end{array}
\right.
\]
Then, we can check that
\[
1\leq\left\Vert \Delta u_{k}\right\Vert _{\frac{N}{2}}^{\frac{N}{2}}%
\leq1+O\left(  \frac{1}{\ln k}\right)  .
\]
Also,
\begin{align*}
\left\Vert u_{k}\right\Vert _{\frac{N}{2}}^{\frac{N}{2}}  &  \leq\omega
_{N-1}\left(  N\beta\left(  N,2\right)  ^{\frac{2}{N}-1}\left(  \ln k\right)
^{-\frac{2}{N}}\right)  ^{\frac{N}{2}}%
{\displaystyle\int\limits_{0}^{1}}
r^{N-1}\ln\frac{1}{r}dr\\
&  +\frac{\omega_{N-1}}{N}\left(  \left[  \frac{1}{\beta\left(  N,2\right)
}\ln k\right]  ^{1-\frac{2}{N}}+\frac{1}{\left(  \frac{\ln k}{k}\right)
^{\frac{2}{N}}}\right)  ^{\frac{N}{2}}\frac{1}{k}\\
&  \leq A\left(  \ln k\right)  ^{-1}+B\left(  \ln k\right)  ^{\frac{N-2}{2}%
}\frac{1}{k}%
\end{align*}
for some constants $A,B>0$.

Let
\[
v_{k}=\frac{u_{k}}{\left\Vert \Delta u_{k}\right\Vert _{\frac{N}{2}}}%
\]
then
\[
\left\Vert \Delta v_{k}\right\Vert _{\frac{N}{2}}=1
\]
and
\[
\left\Vert v_{k}\right\Vert _{\frac{N}{2}}^{\frac{N}{2}}\leq\left\Vert
u_{k}\right\Vert _{\frac{N}{2}}^{\frac{N}{2}}\leq A\left(  \ln k\right)
^{-1}+B\left(  \ln k\right)  ^{\frac{N-2}{2}}\frac{1}{k}.
\]
By the definition of $ATA\left(  \alpha,\beta\right)  ,$ we get%
\begin{align*}
ATA\left(  \alpha,\beta\right)   &  \geq\frac{1}{\left\Vert v_{k}\right\Vert
_{\frac{N}{2}}^{\frac{N}{2}\left(  1-\frac{\beta}{N}\right)  }}%
{\displaystyle\int\limits_{\mathbb{R}^{N}}}
\phi_{N,2}\left(  \alpha\left(  1-\frac{\beta}{N}\right)  \left\vert
v_{k}\right\vert ^{\frac{N}{N-2}}\right)  \frac{dx}{\left\vert x\right\vert
^{\beta}}\\
&  \geq\frac{1}{\left\Vert v_{k}\right\Vert _{\frac{N}{2}}^{\frac{N}{2}\left(
1-\frac{\beta}{N}\right)  }}%
{\displaystyle\int\limits_{\left\vert x\right\vert \leq\left(  \frac{1}%
{k}\right)  ^{\frac{1}{N}}}}
\phi_{N,2}\left(  \alpha\left(  1-\frac{\beta}{N}\right)  \left\vert
v_{k}\right\vert ^{\frac{N}{N-2}}\right)  \frac{dx}{\left\vert x\right\vert
^{\beta}}\\
&  \geq C\frac{\exp\left(  \frac{\alpha}{\beta\left(  N,2\right)  }\left(
1-{\frac{\beta}{N}}\right)  \left(  \frac{1}{\left\Vert \Delta u_{k}%
\right\Vert _{\frac{N}{2}}^{\frac{2}{N-2}}}-{\frac{\beta(N,2)}{\alpha}%
}\right)  \ln k\right)  }{\left(  A\left(  \ln k\right)  ^{-1}+B\left(  \ln
k\right)  ^{\frac{N-2}{2}}\frac{1}{k}\right)  ^{1-\frac{\beta}{N}}}%
\end{align*}
Note that when $k$ (independent of $\alpha$) is large
\[
\frac{1}{\left\Vert \Delta u_{k}\right\Vert _{\frac{N}{2}}^{\frac{2}{N-2}}%
}-{\frac{\beta(N,2)}{\alpha}}\approx1-{{\frac{\beta(N,2)}{\alpha}}}.
\]
So we have
\[
ATA(\alpha,\beta)\gtrsim\exp{\left\{  \left(  1-{\frac{\beta}{N}}\right)
\left(  {\frac{\alpha}{\beta(N,2)}}-1\right)  \ln k\right\}  }\cdot\left(  \ln
k\right)  ^{1-{\frac{\beta}{N}}}%
\]
When $\alpha$ is close enough to $\beta(N,2)$, we are able to choose $k$ large
enough as required before such that
\[
\ln k\approx\frac{1}{1-{\frac{\alpha}{\beta(N,2)}}}%
\]
or
\[
\left(  1-{\frac{\beta}{N}}\right)  \left(  {\frac{\alpha}{\beta(N,2)}%
}-1\right)  \ln k\approx1.
\]
Then
\[
ATA(\alpha,\beta)\gtrsim C\cdot\left(  \frac{1}{1-{\frac{\alpha}{\beta(N,2)}}%
}\right)  ^{1-{\frac{\beta}{N}}}\approx\left(  \frac{1}{1-\left(
{\frac{\alpha}{\beta(N,2)}}\right)  ^{\frac{N-2}{2}}}\right)  ^{1-{\frac
{\beta}{N}}}%
\]
when $\alpha$ is close enough to $\beta(N,2)$.
\end{proof}

We now offer another proof to Theorem \ref{A} using the improved sharp
subcritical Adams inequality (\ref{1.3.3}). \ 

\begin{proof}
[Proof of Theorem \ref{A}]Assume $0<b\leq\frac{N}{2}$. Let $u\in W^{2,\frac
{N}{2}}\left(
\mathbb{R}
^{N}\right)  \setminus\left\{  0\right\}  :\left\Vert \Delta u\right\Vert
_{\frac{N}{2}}^{a}+\left\Vert u\right\Vert _{\frac{N}{2}}^{b}\leq1.$ Assume
that%
\[
\left\Vert \Delta u\right\Vert _{\frac{N}{2}}=\theta\in\left(  0,1\right)
;~\left\Vert u\right\Vert _{\frac{N}{2}}^{b}\leq1-\theta^{a}.
\]
If $\frac{1}{4}<\theta<1$, then we set%
\begin{align*}
v\left(  x\right)   &  =\frac{u\left(  \lambda x\right)  }{\theta}\\
\lambda &  =\frac{\left(  1-\theta^{a}\right)  ^{\frac{1}{2b}}}{\theta
^{\frac{1}{2}}}>0.
\end{align*}
Hence%
\begin{align*}
\left\Vert \Delta v\right\Vert _{\frac{N}{2}}  &  =\frac{\left\Vert \Delta
u\right\Vert _{\frac{N}{2}}}{\theta}=1;\\
\left\Vert v\right\Vert _{\frac{N}{2}}^{\frac{N}{2}}  &  =\int_{%
\mathbb{R}
^{N}}\left\vert v\right\vert ^{\frac{N}{2}}dx=\frac{1}{\theta^{\frac{N}{2}}%
}\int_{%
\mathbb{R}
^{N}}\left\vert u\left(  \lambda x\right)  \right\vert ^{\frac{N}{2}}%
dx=\frac{1}{\theta^{\frac{N}{2}}\lambda^{N}}\left\Vert u\right\Vert _{\frac
{N}{2}}^{\frac{N}{2}}\leq\frac{\left(  1-\theta^{a}\right)  ^{\frac{N}{2b}}%
}{\theta^{\frac{N}{2}}\lambda^{N}}=1.
\end{align*}
By Theorem \ref{ATA}, we get%
\begin{align*}
&  \int_{%
\mathbb{R}
^{N}}\frac{\phi_{N,2}\left(  \beta\left(  N,2\right)  \left(  1-\frac{\beta
}{N}\right)  \left\vert u\right\vert ^{\frac{N}{N-2}}\right)  }{\left\vert
x\right\vert ^{\beta}}dx=\int_{%
\mathbb{R}
^{N}}\frac{\phi_{N,2}\left(  \beta\left(  N,2\right)  \left(  1-\frac{\beta
}{N}\right)  \left\vert u\left(  \lambda x\right)  \right\vert ^{\frac{N}%
{N-2}}\right)  }{\left\vert \lambda x\right\vert ^{\beta}}d\left(  \lambda
x\right) \\
&  \leq\lambda^{N-\beta}\int_{%
\mathbb{R}
^{N}}\frac{\phi_{N,2}\left(  \theta^{\frac{N}{N-2}}\beta\left(  N,2\right)
(1-\frac{\beta}{N})\left\vert v\right\vert ^{N/(N-2)}\right)  }{\left\vert
x\right\vert ^{\beta}}dx\\
&  \leq\lambda^{N-\beta}ATA\left(  \theta^{\frac{N}{N-2}}\beta\left(
N,2\right)  ,\beta\right)  \leq\left(  \frac{\left(  1-\theta^{a}\right)
^{\frac{1}{2b}}}{\theta^{\frac{1}{2}}}\right)  ^{N-\beta}\frac{C\left(
N,\beta\right)  }{\left[  1-\left(  \frac{\theta^{\frac{N}{N-2}}\beta\left(
N,2\right)  }{\beta\left(  N,2\right)  }\right)  ^{\frac{N-2}{2}}\right]
^{1-\frac{\beta}{N}}}\\
&  \leq\frac{\left(  \left(  1-\theta^{a}\right)  ^{\frac{N}{2b}}\right)
^{1-\frac{\beta}{N}}}{\left(  1-\theta^{\frac{N}{2}}\right)  ^{1-\frac{\beta
}{N}}}C\left(  N,\beta\right)  \leq C\left(  N,\beta,a,b\right)  \text{ since
}b\leq\frac{N}{2}\text{.}%
\end{align*}
If $0<\theta\leq\frac{1}{4}$, then with
\[
v\left(  x\right)  =2^{2}u\left(  2x\right)  ,
\]
we have%
\begin{align*}
\left\Vert \Delta v\right\Vert _{\frac{N}{2}}  &  =4\left\Vert \Delta
u\right\Vert _{\frac{N}{2}}\leq1\\
\left\Vert v\right\Vert _{\frac{N}{2}}  &  \leq1.
\end{align*}
By Theorem \ref{ATA}:%
\begin{align*}
&  \int_{%
\mathbb{R}
^{N}}\frac{\phi_{N,2}\left(  \beta\left(  N,2\right)  \left(  1-\frac{\beta
}{N}\right)  \left\vert u\right\vert ^{\frac{N}{N-2}}\right)  }{\left\vert
x\right\vert ^{\beta}}dx\leq4^{N}\int_{%
\mathbb{R}
^{N}}\frac{\phi_{N}\left(  \frac{\beta\left(  N,2\right)  (1-\frac{\beta}{N}%
)}{4^{\frac{N}{N-2}}}\left\vert v\right\vert ^{N/(N-2)}\right)  }{\left\vert
x\right\vert ^{\beta}}dx\\
&  \leq C\left(  N,\beta\right)  .
\end{align*}

We now also consider the Adams' test functions as in the proof of Theorem
\ref{ATA}. Let $\beta>\beta\left(  N,2\right)  $. Set
\begin{align*}
~w_{r}(\left\vert x\right\vert )  &  =\lambda_{r}\frac{u_{r}\left(  \left\vert
x\right\vert \right)  }{\left\Vert \Delta u_{r}\right\Vert _{\frac{N}{2}}%
}\text{ where }\lambda_{r}\in\left(  0,1\right)  \text{ is a solution of
}\lambda_{r}^{a}+\frac{\lambda_{r}^{b}\left\Vert u_{r}\right\Vert _{\frac
{N}{2}}^{b}}{\left\Vert \Delta u_{r}\right\Vert _{\frac{N}{2}}^{b}}=1.\\
\lambda_{r}  &  \rightarrow_{r\rightarrow0^{+}}1.
\end{align*}
Then
\[
\left\Vert \Delta w_{r}\right\Vert _{\frac{N}{2}}^{a}+\left\Vert
w_{r}\right\Vert _{\frac{N}{2}}^{b}=1
\]
and
\begin{align*}
&  \lim_{r\rightarrow0^{+}}\int_{%
\mathbb{R}
^{N}}\frac{\phi_{N,2}\left(  \beta\left(  1-\frac{\beta}{N}\right)  \left\vert
w_{r}\right\vert ^{\frac{N}{N-2}}\right)  }{\left\vert x\right\vert ^{\beta}%
}dx\\
\geq &  \lim_{r\rightarrow0^{+}}%
{\displaystyle\int\limits_{B_{r}}}
\phi_{N,2}\left(  \frac{\beta\left(  1-\frac{\beta}{N}\right)  \lambda
_{r}^{\frac{N}{N-2}}\left\vert u_{r}\right\vert ^{\frac{N}{N-2}}}{\left\Vert
\Delta u_{r}\right\Vert _{\frac{N}{2}}^{\frac{N}{N-2}}}\right)  \frac
{dx}{\left\vert x\right\vert ^{\beta}}\\
\geq &  \lim_{r\rightarrow0^{+}}\omega_{N-1}\frac{r^{N-\beta}}{N-\beta}%
\phi_{N,2}\left(  \frac{\beta}{\beta\left(  N,2\right)  }\frac{\left(
N-\beta\right)  \log\left(  \frac{1}{r}\right)  }{\left[  1+2\varepsilon
\left(  \left\Vert \psi^{\prime}\right\Vert _{\infty}+O\left(  \frac{1}%
{\log\left(  \frac{1}{r}\right)  }\right)  \right)  ^{\frac{N}{2}}\right]
^{\frac{2}{N-2}}}\right) \\
&  \rightarrow\infty\text{ as }r\rightarrow0^{+}\text{ if we choose
}\varepsilon\text{ small enough.\ }%
\end{align*}

It now remains to show that%
\[
A_{a,b}\left(  \beta\right)  =\sup_{\alpha\in\left(  0,\beta\left(
N,2\right)  \right)  }\left(  \frac{1-\left(  \frac{\alpha}{\beta\left(
N,2\right)  }\right)  ^{\frac{N-2}{N}a}}{\left(  \frac{\alpha}{\beta\left(
N,2\right)  }\right)  ^{\frac{N-2}{N}b}}\right)  ^{\frac{N-\beta}{2b}%
}ATA\left(  \alpha,\beta\right)  .
\]
By (\ref{2.2}):%
\[
\sup_{\alpha\in\left(  0,\beta\left(  N,2\right)  \right)  }\left(
\frac{1-\left(  \frac{\alpha}{\beta\left(  N,2\right)  }\right)  ^{\frac
{N-2}{N}a}}{\left(  \frac{\alpha}{\beta\left(  N,2\right)  }\right)
^{\frac{N-2}{N}b}}\right)  ^{\frac{N-\beta}{2b}}ATA\left(  \alpha
,\beta\right)  \leq A_{a,b}\left(  \beta\right)  .
\]
Now, let $\left(  u_{n}\right)  $ be the maximizing sequence of $A_{a,b}%
\left(  \beta\right)  $, i.e., $u_{n}\in W^{2,\frac{N}{2}}\left(
\mathbb{R}
^{N}\right)  \setminus\left\{  0\right\}  :\left\Vert \Delta u_{n}\right\Vert
_{\frac{N}{2}}^{a}+\left\Vert u_{n}\right\Vert _{\frac{N}{2}}^{b}\leq1$ and
\[
\int_{%
\mathbb{R}
^{N}}\phi_{N,2}\left(  \beta\left(  N,2\right)  \left(  1-\frac{\beta}%
{N}\right)  \left\vert u_{n}\right\vert ^{\frac{N}{N-2}}\right)  \frac
{dx}{\left\vert x\right\vert ^{\beta}}\rightarrow_{n\rightarrow\infty}%
A_{a,b}\left(  \beta\right)  .
\]
We define a new sequence:%
\begin{align*}
v_{n}\left(  x\right)   &  =\frac{u\left(  \lambda_{n}x\right)  }{\left\Vert
\Delta u_{n}\right\Vert _{\frac{N}{2}}}\\
\lambda_{n}  &  =\left(  \frac{1-\left\Vert \Delta u_{n}\right\Vert _{\frac
{N}{2}}^{a}}{\left\Vert \Delta u_{n}\right\Vert _{\frac{N}{2}}^{b}}\right)
^{\frac{1}{2b}}>0.
\end{align*}
Hence%
\[
\left\Vert \Delta v_{n}\right\Vert _{\frac{N}{2}}=1\text{ and }\left\Vert
v_{n}\right\Vert _{\frac{N}{2}}\leq1.
\]
Also,
\begin{align*}
&  \int_{%
\mathbb{R}
^{N}}\phi_{N,2}\left(  \beta\left(  N,2\right)  \left(  1-\frac{\beta}%
{N}\right)  \left\vert u_{n}\right\vert ^{\frac{N}{N-2}}\right)  \frac
{dx}{\left\vert x\right\vert ^{\beta}}\\
&  =\lambda_{n}^{N-\beta}\int_{%
\mathbb{R}
^{N}}\frac{\phi_{N,2}\left(  \left\Vert \Delta u_{n}\right\Vert _{\frac{N}{2}%
}^{N/(N-2)}\beta\left(  N,2\right)  (1-\frac{\beta}{N})\left\vert
v_{n}\right\vert ^{N/(N-2)}\right)  }{\left\vert x\right\vert ^{\beta}}dx\\
&  \leq\lambda_{n}^{N-\beta}ATA\left(  \left\Vert \Delta u_{n}\right\Vert
_{\frac{N}{2}}^{N/(N-2)}\beta\left(  N,2\right)  ,\beta\right)  \leq
\sup_{\alpha\in\left(  0,\beta\left(  N,2\right)  \right)  }\left(
\frac{1-\left(  \frac{\alpha}{\beta\left(  N,2\right)  }\right)  ^{\frac
{N-2}{N}a}}{\left(  \frac{\alpha}{\beta\left(  N,2\right)  }\right)
^{\frac{N-2}{N}b}}\right)  ^{\frac{N-\beta}{2b}}ATA\left(  \alpha
,\beta\right)  .
\end{align*}
Now, we assume that there is some $b>\frac{N}{2}$ such that $A_{a,b}\left(
\beta\right)  <\infty$. Then
\[
A_{a,b}\left(  \beta\right)  =\sup_{\alpha\in\left(  0,\beta\left(
N,2\right)  \right)  }\left(  \frac{1-\left(  \frac{\alpha}{\beta\left(
N,2\right)  }\right)  ^{\frac{N-2}{N}a}}{\left(  \frac{\alpha}{\beta\left(
N,2\right)  }\right)  ^{\frac{N-2}{N}b}}\right)  ^{\frac{N-\beta}{2b}%
}ATA\left(  \alpha,\beta\right)
\]
and so%
\[
\overline{\lim}_{\alpha\uparrow\beta\left(  N,2\right)  }\left(
\frac{1-\left(  \frac{\alpha}{\beta\left(  N,2\right)  }\right)  ^{\frac
{N-2}{N}a}}{\left(  \frac{\alpha}{\beta\left(  N,2\right)  }\right)
^{\frac{N-2}{N}b}}\right)  ^{\frac{N-\beta}{2b}}ATA\left(  \alpha
,\beta\right)  <\infty.
\]
Also, by Theorem \ref{ATA}:
\begin{equation}
\overline{\lim}_{\alpha\uparrow\beta\left(  N,2\right)  }\left(
\frac{1-\left(  \frac{\alpha}{\beta\left(  N,2\right)  }\right)  ^{\frac
{N-2}{N}a}}{\left(  \frac{\alpha}{\beta\left(  N,2\right)  }\right)
^{\frac{N-2}{N}b}}\right)  ^{\frac{N-\beta}{N}}ATA\left(  \alpha,\beta\right)
>0, \label{4.1.1}%
\end{equation}
Hence:%
\[
\underline{\lim}_{\alpha\uparrow\beta\left(  N,2\right)  }\frac{\left(
1-\left(  \frac{\alpha}{\beta\left(  N,2\right)  }\right)  ^{\frac{N-2}{N}%
a}\right)  ^{\frac{N-\beta}{2b}}}{\left(  1-\left(  \frac{\alpha}{\beta\left(
N,2\right)  }\right)  ^{\frac{N-2}{N}a}\right)  ^{\frac{N-\beta}{N}}}>0
\]
which is impossible since $b>\frac{N}{2}$. The proof is now completed.
\end{proof}

\subsection{Adams inequalities on $W^{\gamma,\frac{N}{\gamma}}\left(
\mathbb{R}
^{N}\right)  $-Proof of Theorem \ref{generalA}}

Let $u\in W^{\gamma,p}\left(
\mathbb{R}
^{N}\right)  \setminus\left\{  0\right\}  :\left\Vert \left(  -\Delta\right)
^{\frac{\gamma}{2}}u\right\Vert _{p}^{a}+\left\Vert u\right\Vert _{p}^{b}%
\leq1.$ We set%
\[
\left\Vert \left(  -\Delta\right)  ^{\frac{\gamma}{2}}u\right\Vert _{p}%
=\theta\in\left(  0,1\right)  ;~\left\Vert u\right\Vert _{p}^{b}\leq
1-\theta^{a}.
\]
If $\frac{1}{2^{\gamma}}<\theta<1$, then by define a new function%
\begin{align*}
v\left(  x\right)   &  =\frac{u\left(  \lambda x\right)  }{\theta}\\
\lambda &  =\frac{\left(  1-\theta^{a}\right)  ^{\frac{1}{\gamma b}}}%
{\theta^{\frac{1}{\gamma}}}>0.
\end{align*}
we get%
\[
\left(  -\Delta\right)  ^{\frac{\gamma}{2}}v\left(  x\right)  =\frac
{\lambda^{\gamma}}{\theta}\left(  \left(  -\Delta\right)  ^{\frac{\gamma}{2}%
}u\right)  \left(  \lambda x\right)  .
\]
Hence
\begin{align*}
\left\Vert \left(  -\Delta\right)  ^{\frac{\gamma}{2}}v\right\Vert _{p} &
=\frac{\left\Vert \left(  -\Delta\right)  ^{\frac{\gamma}{2}}u\right\Vert
_{p}}{\theta}=1;\\
\left\Vert v\right\Vert _{p}^{p} &  =\int_{%
\mathbb{R}
^{N}}\left\vert v\right\vert ^{p}dx=\frac{1}{\theta^{p}}\int_{%
\mathbb{R}
^{N}}\left\vert u\left(  \lambda x\right)  \right\vert ^{p}dx=\frac{1}%
{\theta^{p}\lambda^{N}}\left\Vert u\right\Vert _{p}^{p}\leq\frac{\left(
1-\theta^{a}\right)  ^{\frac{p}{b}}}{\theta^{p}\lambda^{N}}=1.
\end{align*}
By the definition of $GATA\left(  \alpha,\beta\right)  $, we get%
\begin{align*}
&  \int_{%
\mathbb{R}
^{N}}\frac{\phi_{N,\gamma}\left(  \beta_{0}\left(  N,\gamma\right)  \left(
1-\frac{\beta}{N}\right)  \left\vert u\right\vert ^{\frac{p}{p-1}}\right)
}{\left\vert x\right\vert ^{\beta}}dx=\int_{%
\mathbb{R}
^{N}}\frac{\phi_{N,\gamma}\left(  \beta_{0}\left(  N,\gamma\right)  \left(
1-\frac{\beta}{N}\right)  \left\vert u\left(  \lambda x\right)  \right\vert
^{\frac{p}{p-1}}\right)  }{\left\vert \lambda x\right\vert ^{\beta}}d\left(
\lambda x\right)  \\
&  \leq\lambda^{N-\beta}\int_{%
\mathbb{R}
^{N}}\frac{\phi_{N,\gamma}\left(  \theta^{\frac{p}{p-1}}\beta_{0}\left(
N,\gamma\right)  \left(  1-\frac{\beta}{N}\right)  \left\vert v\right\vert
^{\frac{p}{p-1}}\right)  }{\left\vert x\right\vert ^{\beta}}dx\\
&  \leq\lambda^{N-\beta}GATA\left(  \theta^{\frac{p}{p-1}}\beta_{0}\left(
N,\gamma\right)  ,\beta\right)  \leq\left(  \frac{\left(  1-\theta^{a}\right)
^{\frac{1}{\gamma b}}}{\theta^{\frac{1}{\gamma}}}\right)  ^{N-\beta}%
\frac{C\left(  N,\beta\right)  }{\left[  1-\left(  \frac{\theta^{\frac{p}%
{p-1}}\beta_{0}\left(  N,\gamma\right)  }{\beta_{0}\left(  N,\gamma\right)
}\right)  ^{\frac{p-1}{p}}\right]  ^{1-\frac{\beta}{N}}}\\
&  \leq\frac{\left(  \left(  1-\theta^{a}\right)  ^{\frac{N}{\gamma b}%
}\right)  ^{1-\frac{\beta}{N}}}{\left(  1-\theta\right)  ^{1-\frac{\beta}{N}}%
}C\left(  N,\beta\right)  \leq C\left(  N,\beta,a,b\right)  \text{ since
}b\leq p\text{.}%
\end{align*}
If $0<\theta\leq\frac{1}{2^{\gamma}}$, then with
\[
v\left(  x\right)  =2^{\gamma}u\left(  2x\right)  ,
\]
we have%
\begin{align*}
\left\Vert \left(  -\Delta\right)  ^{\frac{\gamma}{2}}v\right\Vert _{p} &
=2^{\gamma}\left\Vert \left(  -\Delta\right)  ^{\frac{\gamma}{2}}u\right\Vert
_{p}\leq1\\
\left\Vert v\right\Vert _{p} &  \leq1.
\end{align*}
By the definition of $GATA\left(  \alpha,\beta\right)  :$%
\begin{align*}
&  \int_{%
\mathbb{R}
^{N}}\frac{\phi_{N,\gamma}\left(  \beta_{0}\left(  N,\gamma\right)  \left(
1-\frac{\beta}{N}\right)  \left\vert u\right\vert ^{\frac{p}{p-1}}\right)
}{\left\vert x\right\vert ^{\beta}}dx\leq2^{N}\int_{%
\mathbb{R}
^{N}}\frac{\phi_{N,\gamma}\left(  \frac{\beta_{0}\left(  N,\gamma\right)
}{2^{\gamma\frac{p}{p-1}}}\left(  1-\frac{\beta}{N}\right)  \left\vert
v\right\vert ^{\frac{p}{p-1}}\right)  }{\left\vert x\right\vert ^{\beta}}dx\\
&  \leq C\left(  N,\beta\right)  .
\end{align*}

\textbf{Acknowledgement} The results of this paper have been presented in
invited talks by the second author at Zhejiang University in China in June,
2014 and at the Workshop on Partial Differential Equation and its Applications
at IMS in Singapore in December 2014 and by the first author at the AMS
special session on Geometric inequalities and Nonlinear Partial Differential
Equations in Las Vegas in April 2015. The authors wish to thank the organizers
for invitations.

\end{document}